\documentclass[11pt]{amsart}

\usepackage{amsaddr}
\usepackage{caption}
\usepackage{enumerate}
\usepackage{float}
\usepackage{graphicx}
\usepackage{geometry} 
\usepackage{setspace}
\usepackage{subcaption}
\usepackage{tabu} 

\newcommand{\vx}{\vct{x}}

\newcommand{\vu}{\vct{u}}

\newcommand{\vvu}{\vvct{u}}
\newcommand{\vvU}{\vvct{U}}

\newcommand{\vv}{\vct{v}}
\newcommand{\vvv}{\vvct{v}}

\newcommand{\vw}{\vct{w}}

\newcommand{\vun}{\hat{\vct{n}}}

\newcommand{\vf}{\vct{f}}

\newcommand{\vkappa}{\boldsymbol{\kappa}}

\newcommand{\vo}{\vct{0}}

\newcommand{\tsigma}{{\boldsymbol{\sigma}}}

\newcommand{\tzero}{\ten{0}}

\newcommand{\Ch}{C}

\newcommand{\im}{\hat{\i}}

\newcommand{\av}[1]{ \{\!\!\{#1\}\!\!\} }
\newcommand{\ju}[1]{ [\![#1]\!] }

\newcommand{\vct}[1]{{\mathbf{#1}}}

\newcommand{\vvct}[1]{\underline{\mathbf{#1}}}
\newcommand{\ten}[1]{{\mathbf{#1}}}

\newtheorem{thm}{Theorem}[section]

\newtheorem{rem}[thm]{Remark}

\newtheorem{case}{Case}

\title[Dispersion Properties of Explicit Finite Element Methods]{Dispersion Properties of Explicit Finite Element Methods for Wave Propagation Modelling on Tetrahedral Meshes*}
\author{S. Geevers$^1$, W.A. Mulder$^{2,3}$ \and J.J.W. van der Vegt$^1$}
\address{1. Department of Applied Mathematics, University of Twente, Enschede, the Netherlands (e-mail: s.geevers@utwente.nl, j.j.w.vandervegt@utwente.nl)}
\address{2. Shell Global Solutions International BV (e-mail: wim.mulder@shell.com)}
\address{3. Delft University of Technology}
\thanks{*This work was funded by the Shell Global Solutions International B.V. under contract no.~PT45999.}

\begin{document}
\maketitle

\begin{abstract}
We analyse the dispersion properties of two types of explicit finite element methods for modelling acoustic and elastic wave propagation on tetrahedral meshes, namely mass-lumped finite element methods and symmetric interior penalty discontinuous Galerkin methods, both combined with a suitable Lax--Wendroff time integration scheme. The dispersion properties are obtained semi-analytically using standard Fourier analysis. Based on the dispersion analysis, we give an indication of which method is the most efficient for a given accuracy, how many elements per wavelength are required for a given accuracy, and how sensitive the accuracy of the method is to poorly shaped elements.
\end{abstract}

\section{Introduction}
\label{sec:introduction}
Realistic wave propagation problems often involve large three-dimensional domains consisting of heterogeneous materials with complex geometries and sharp interfaces. Solving such problems requires a numerical method that is efficient in terms of computation time and is flexible enough to capture the effect of a complex geometry. 

Standard finite difference methods fall short, since they rely on Cartesian grids that cannot efficiently capture the effect of complex interfaces and boundary layers. Finite element methods can overcome this problem when the elements are aligned with those surfaces. However, the accuracy of the finite element method quickly deteriorates when the elements are poorly shaped or are poorly aligned with the geometry. Obtaining a high quality mesh is therefore quintessential. While both hexahedral and tetrahedral elements are commonly used for three-dimensional problems, tetrahedral elements have a big advantage in this respect, since they offer more geometric flexibility and since robust tetrahedral mesh generators based on the Delaunay criterion are available \cite{owen98, si15}.

Apart from the construction of a high-quality mesh, finite element methods for wave propagation problems also require a (block)-diagonal mass matrix to enable explicit time-stepping. A diagonal mass matrix can be obtained with nodal basis functions and a quadrature rule, if the quadrature points coincide with the basis function nodes. This technique is known as mass-lumping. For quadrilaterals and hexahedra, mass-lumping is achieved by combining tensor-product basis functions with a Gauss-Lobatto quadrature rule, resulting in a scheme known as the spectral element method \cite{patera84, seriani94, komatitsch98}. For triangles and tetrahedra, an efficient linear mass-lumped scheme is obtained by combining standard Lagrangian basis functions with a Newton--Cotes quadrature rule. For higher-degree triangles and tetrahedra, however, this approach results in an unstable, unsolvable, or inaccurate scheme. To remain accurate and stable, the space of the triangle or tetrahedron is enriched with higher-degree bubble functions. This approach has led to accurate mass-lumped triangles of degree 2 \cite{fried75}, 3 \cite{cohen95}, 4 \cite{mulder96}, 5 \cite{chin99}, 6 \cite{mulder13}, 7-9 \cite{liu17,cui17} and tetrahedra of degree 2 \cite{mulder96} and 3 \cite{chin99}.


Another way to obtain a (block)-diagonal mass matrix is by using discontinuous basis functions. The resulting schemes are known as Discontinuous Galerkin (DG) methods. The first DG methods for wave propagation problems were based on a first-order formulation of the wave equation \cite{reed73, cockburn89}. In \cite{riviere03} and \cite{grote06}, DG methods  were introduced that were based on the original second-order formulation of the wave problem. The advantage of finite element methods based on the second-order formulation is that they do not need to compute or store the auxiliary variables that appear in the first-order formulation.  Moreover, they can be combined with a leap-frog or higher-order Lax--Wendroff time integration scheme that only requires $K$ stages for a $2K$-order accuracy. We focus on the symmetric interior penalty discontinuous Galerkin (SIPDG) method, presented and analysed in \cite{grote06}, which is based on the second-order formulation of the wave problem and which also remains energy-conservative on the discrete level. To remain accurate and stable, face integrals and interior penalty parameters are added to the discrete operator. We consider two choices for the penalty parameter: the penalty term derived in \cite{mulder14}, based on the trace inequality of \cite{warburton03}, and a recently developed sharper estimate \cite{geevers17}, based on a more involved trace inequality. 

To effectively apply these methods, it is crucial to know the required mesh resolution for a given accuracy. It is also useful to know which method is the most efficient for a given accuracy and how the mesh quality and material parameters, such as the P/S-wave velocity ratio for elastic waves, affect the accuracy. A practical and common measure for the accuracy of these type of methods is the amount of numerical dispersion and dissipation. In this paper, we will focus mainly on the numerical dispersion, since the methods we consider are all energy-conservative and therefore do not suffer from numerical dissipation. We do, however, also investigate the spurious modes that appear when projecting a physical wave onto the discrete space.

The dispersion properties of DG methods based on the first-order formulation of the wave problem have already been analysed for Cartesian meshes \cite{hu99, ainsworth04}, triangles \cite{hu99, lisitsa16}, and tetrahedra \cite{kaser08}. For the SIPDG method, these properties have already been analysed for Cartesian meshes in \cite{ainsworth06, deBasabe08} and triangles in \cite{antonietti16} and for the mass-lumped finite element method this has already been analysed for quadrilaterals and hexahedra in \cite{mulder99,cohen02,deBasabe07} and for triangles in \cite{liu12}. However, a dispersion analysis of the mass-lumped finite element and SIPDG methods for tetrahedra is, to the best of our knowledge, still missing, even though most realistic wave problems involve three-dimensional domains for which tetrahedral elements are particularly suitable. In this paper, we therefore present an extensive dispersion analysis of these methods for tetrahedra. This analysis is based on standard Fourier analysis. We use the analysis to obtain estimates for the required number of elements per wavelength and estimate the computational cost to obtain an indication of which method is the most efficient for a given accuracy. We consider both acoustic and elastic waves and also look at the effect of poorly shaped elements and high P/S-wave velocity ratios on the accuracy of the methods.


This paper is organised as follows: in Section \ref{sec:tenNotation}, we introduce the tensor notation used in this paper. The acoustic and elastic wave equations are presented in Section \ref{sec:genHypModel} and the mass-lumped and discontinuous Galerkin finite element methods are presented in Section \ref{sec:FEM}. In Section \ref{sec:dispAn}, we explain how we analyse the dispersion properties of these methods. The results of this analysis are presented in Section \ref{sec:results} and the main conclusions are summarised in Section \ref{sec:conclusions}.

\section{Some Tensor Notation}
\label{sec:tenNotation}
Before we present the acoustic and elastic wave equations, we explain the tensor notation that we use throughout this paper. We let the dot product of two tensors denote the summation over the last index of the left and first index of the right tensor. For the double dot product we also sum over the last-but-one index of the left and second index of the right tensor. A concatenation of two tensors denotes the standard tensor product. To give some examples, let $\vun\in\mathbb{R}^d, \vu\in\mathbb{R}^m$ be two vectors, $\tsigma\in\mathbb{R}^{d\times m}$ a second-order tensor, and $C\in\mathbb{R}^{d\times m\times m\times d}$ a fourth-order tensor.  Then
\begin{align*}
[\vun\vu]_{ij} &:= \hat{n}_{i} u_{j},   &[\tsigma\cdot\vu]_{i} &:= \sum_{l=1}^m \sigma_{il}u_{l},   \\
[C:\tsigma]_{ij} &:= \sum_{k=1}^d\sum_{l=1}^m C_{ijlk}\sigma_{kl},  &[\vun\cdot C]_{qji} &:= \sum_{k=1}^d \hat{n}_kC_{kqji},
\end{align*}
for all $i=1,\dots,d$ and $j,q=1,\dots,m$.

In the next section we will use this tensor notation to present the acoustic and isotropic elastic wave equations.

\section{The Acoustic and Isotropic Elastic Wave Equations}
\label{sec:genHypModel}
Let $\Omega\subset\mathbb{R}^3$ be a three-dimensional open domain with a Lipschitz boundary $\partial\Omega$, and let $(0,T)$ be the time domain. Also, let $\{\Gamma_d,\Gamma_n\}$ be a partition of $\partial\Omega$, corresponding to Dirichlet and von Neumann boundary conditions, respectively. We define the following linear hyperbolic problem:
\begin{subequations}
\begin{align}
\rho\partial_t^2\vu &= \nabla\cdot C:\nabla \vu + \vf &&\text{in }\Omega\times(0,T),  \label{eq:model_a} \\
C:\vun\vu &=\tzero &&\text{on }\Gamma_d\times(0,T),  \label{eq:model_b}\\
\vun\cdot C:\nabla\vu &= \vo &&\text{on }\Gamma_n\times(0,T), \\
\vu|_{t=0} &= \vu_0 &&\text{in }\Omega, \\
\partial_t\vu|_{t=0} &= \vv_0 && \text{in }\Omega,
\end{align}
\label{eq:model}%
\end{subequations}
where $\vu:\Omega\times(0,T)\rightarrow\mathbb{R}^m$ is a vector of $m$ variables that are to be solved, $\nabla$ is the gradient operator, $\rho:\Omega\rightarrow\mathbb{R}^+$ is a positive scalar field, $C:\Omega\rightarrow\mathbb{R}^{3\times m\times m\times 3}$ a fourth-order tensor field, $\vf:\Omega\times(0,T)\rightarrow\mathbb{R}^{m}$ the source field, and $\vun:\partial\Omega\rightarrow\mathbb{R}^3$ the outward pointing normal unit vector. 

By choosing the appropriate tensor and scalar field we can obtain the acoustic wave equation and the isotropic elastic wave equations.

\begin{case}[Isotropic elastic wave equations]
To obtain the isotropic elastic wave equations, set $m=3$ and
\begin{align*}
C_{ijqp} = \lambda\delta_{ij}\delta_{pq} + \mu(\delta_{ip}\delta_{jq} + \delta_{iq}\delta_{jp}),
\end{align*}
for $i,j,p,q=1,2,3$, where $\delta$ is the Kronecker delta. Equation (\ref{eq:model_a}) then becomes
\begin{align*}
\rho\partial_t^2 \vu = \nabla\lambda(\nabla\cdot\vu) + \nabla\cdot\mu(\nabla\vu + \nabla\vu^t) + \vf, 
\end{align*}
where $\vu:\Omega\times(0,T)\rightarrow\mathbb{R}^3$ is the displacement field, $\rho:\Omega\rightarrow\mathbb{R}^+$ is the mass density, $\lambda, \mu:\Omega\rightarrow\mathbb{R}^+$ are the Lam\'e parameters, and $\vf:\Omega\times(0,T)\rightarrow\mathbb{R}^{3}$ is the external volume force. The superscript $t$ denotes the transposed.
\label{case:elasticWave}
\end{case}

\begin{case}[Acoustic wave equation]
To obtain the acoustic wave equation, set $m=1$, $u= p$, $\rho=(\tilde\rho \tilde{c}^2)^{-1}$, and
\begin{align*}
C_{i11j} := \frac{1}{\tilde{\rho}}\delta_{ij},
\end{align*}
for $i,j=1,2,3$, where $\delta$ is the Kronecker delta. Equation (\ref{eq:model_a}) then becomes
\begin{align*}
\frac{1}{\tilde\rho \tilde{c}^2}\partial_t^2 p = \nabla\cdot\frac{1}{\tilde\rho}\nabla p + f,
\end{align*}
where $p:\Omega\times(0,T)\rightarrow\mathbb{R}$ is the pressure field, $\tilde\rho:\Omega\rightarrow\mathbb{R}^+$ the mass density, $\tilde c:\Omega\rightarrow\mathbb{R}^+$ the acoustic velocity field, and $f=\nabla\cdot(\tilde{\rho}^{-1}\tilde\vf)$ the source term with $\tilde{\vf}:\Omega\times(0,T)\rightarrow\mathbb{R}^{3}$ the external volume force. 
\label{case:acousticWave}
\end{case} 

These equations can be solved with the finite element methods described in the next section.

\section{The Discontinuous Galerkin and Mass-Lumped Finite Element Method}
\label{sec:FEM}
\subsection{The Classical Finite Element Method}
Let $\mathcal{T}_h$ be a tetrahedral tessellation of $\Omega$, with $h$ denoting the radius of the smallest sphere that can contain each element and let $\mathcal{U}_h$ be the finite element space consisting of continuous element-wise polynomial basis functions satisfying boundary condition (\ref{eq:model_b}). The classical conforming finite element formulation of (\ref{eq:model}) is finding $\vu_h:[0,T]\rightarrow \mathcal{U}_h$ such that $\vu_h|_{t=0}=\Pi_{h}\vu_0$, $\partial_t\vu_h|_{t=0}=\Pi_{h}\vv_0$ and
\begin{align}
(\rho\partial_t^2\vu_h,\vw) + a(\vu_h,\vw) &= (\vf,\vw), && \vw\in \mathcal{U}_h, t\in[0,T], 
\label{eq:FEM}%
\end{align}
where $(\cdot,\cdot)$ denotes the standard $L^2$ inner product, $\Pi_{h}:L^2(\Omega)\rightarrow \mathcal{U}_h$ denotes the weighted $L^2$-projection operator defined such that $(\rho\Pi_{h}\vu,\vw)=(\rho\vu,\vw)$ for all $\vw\in \mathcal{U}_h$, and $a:H^1(\Omega)^m\times H^1(\Omega)^m\rightarrow \mathbb{R}$ is the (semi)-elliptic operator given by
\begin{align*}
a(\vu,\vw) &:= \int_{\Omega} (\nabla\vu)^t : C : \nabla\vw \;dx.
\end{align*}

Let $\{\vw^{(i)}\}_{i=1}^{n}$ be the set of basis functions spanning $\mathcal{U}_h$, and let, for any $\vu\in L^2(\Omega)^m$, the vector $\vvu\in\mathbb{R}^n$ be defined such that $\sum_{i=1}^n \underline{u}_i\vw^{(i)}=\Pi_{h}\vu$. Also, let $M,A\in\mathbb{R}^{n\times n}$ be the mass matrix and stiffness matrix, respectively, defined by $M_{ij}:=(\rho\vw^{(i)},\vw^{(j)})$ and $A_{ij}:=a(\vw^{(i)},\vw^{(j)})$, and let $\vvct{f}^*:[0,T]\rightarrow\mathbb{R}^n$ be given by $\underline{f}_i^*:=(\vf,\vw^{(i)})$. The finite element method can then be formulated as finding $\vvu_h:[0,T]\rightarrow\mathbb{R}^n$ such that $\vvu_h|_{t=0}=\underline{\vu_0}$, $\partial_t\vvu_h|_{t=0} = \underline{\vv_0}$, and
\begin{align}
M\partial_t^2\vvu_h + A\vvu_h = \vvct{f}^*, &&t\in[0,T].
\label{eq:ODE}
\end{align}

The main drawback of the classical conforming finite element approach is that when an explicit time integration scheme is applied, a system of equations of the form $M\vvct{x}=\vvct{b}$ needs to be solved at every time step, with $M$ not (block)-diagonal. For large-scale problems, this results in a very inefficient time stepping scheme. This problem can be circumvented by lumping the mass matrix into a diagonal matrix or by using discontinuous basis functions.

\subsection{Mass-Lumping}
When using nodal basis functions, the mass matrix can be lumped into a diagonal matrix by taking the sum over each row. This is equivalent to replacing the inner product $(\cdot,\cdot)$ by $(\cdot,\cdot)_h^{(L)}$, in which the element integrals are approximated by a quadrature rule with quadrature points that coincide with the nodes of the basis functions. We can write
\begin{align*}
(\vu,\vw)_h^{(L)} = \sum_{e\in\mathcal{T}_h} \sum_{\vx\in\mathcal{Q}_e} \omega_{e,\vx}\rho(\vx)\vu(\vx)\cdot\vw(\vx),
\end{align*}
where $\mathcal{Q}_e$ denotes the quadrature points on element $e$ and $\omega_{e,\vx}$ denote the quadrature weights. Let $\{\vx^{(i)}\}_{i=1}^n$ denote the global set of integration points and define $\vw^{(i)}$ to be the nodal basis function corresponding to $\vx^{(i)}$, so $\vw^{(i)}(\vx^{(j)})=\delta_{ij}$, with $\delta$ the Kronecker delta. Then the mass matrix becomes diagonal with entries $M_{ii}=\sum_{e\in\mathcal{T}_{\vx^{(i)}}} \omega_{e,\vx^{(i)}}\rho(\vx^{(i)})$, where $\mathcal{T}_{\vx}$ denotes the set of elements containing or adjacent to $\vx$.

For quadrilaterals and hexahedra, mass-lumping is achieved by using tensor-product basis functions and Gauss-Lobatto integration points. The resulting scheme is known as the spectral element method. For triangles and tetrahedra, mass-lumping is less straight-forward. Combining standard Lagrangian basis functions with a Newton--Cotes quadrature rule results in an efficient mass-lumped scheme for linear tetrahedra, but for higher-degree basis functions, this approach results either in an unstable scheme due to non-positive quadrature weights or in a scheme with a reduced order of convergence. This problem can be resolved by enriching the finite element space with higher-degree bubble functions and by adding integration points to the interior of the elements and faces. For example, by enriching the space of the quadratic tetrahedron with 3 degree-4 face bubble functions and 1 degree-4 interior bubble function, an enriched degree-2 mass-lumped tetrahedron that remains third-order accurate can be obtained  \cite{mulder96}. 

In this paper we will analyse the standard linear mass-lumped finite element method, the mass-lumped finite element method of degree 2 derived in \cite{mulder96}, and the 2 versions of degree 3 mass-lumped finite element methods derived in \cite{chin99}. We will refer to these methods as ML1, ML2, ML3a and ML3b, respectively.

\subsection{The Symmetric Interior Penalty Discontinuous Galerkin Method}
Another way to obtain a (block)-diagonal mass matrix is by allowing the finite element space $\mathcal{U}_h$ to be discontinuous at the faces. When choosing basis functions that have support on only a single element, the mass matrix becomes block-diagonal with each block corresponding to a single element. When using orthogonal basis functions, the mass matrix even becomes strictly diagonal. In order to keep the finite element method stable and consistent with the analytic solution, the elliptic operator needs to be augmented. This can be accomplished with the symmetric interior penalty method \cite{grote06}, where $a$ is replaced by the discrete (semi)-elliptic operator $a_h^{(DG)}:\mathcal{U}_h\times \mathcal{U}_h\rightarrow\mathbb{R}$, given by
\begin{align*}
a_h^{(DG)}(\vu,\vw) &:=  a_h^{(C)}(\vu,\vw) - a_h^{(D)}(\vu,\vw) - a_h^{(D)}(\vw,\vu) + a_h^{(IP)}(\vu,\vw) 
\end{align*}
with
\begin{align*}
a_h^{(C)}(\vu,\vw) &:=  \sum_{e\in\mathcal{T}_h} \int_{e} (\nabla\vu)^t:\Ch:\nabla\vw \;d\vct{x}, \\
a_h^{(D)}(\vu,\vw) &:= \sum_{f\in\mathcal{F}_{h,in}\cup\mathcal{F}_{h,d}} \int_f \ju{\vu}^t:\av{\Ch : \nabla\vw} \;d\vct{s}, \\
a_h^{(IP)}(\vu,\vw) &:= \sum_{f\in\mathcal{F}_{h,in}\cup\mathcal{F}_{h,d}} \int_{f} \ju{\vu}^t:\av{\alpha_h \Ch}:\ju{\vw} \;d\vct{s},
\end{align*}
where $\mathcal{F}_{h,in}$ and $\mathcal{F}_{h,d}$ are the internal faces and boundary faces on $\Gamma_d$, respectively, $\alpha_h\in\bigotimes_{e\in\mathcal{T}} L^{\infty}(\partial e)$ is the penalty function, and $\av{\cdot}$, $\ju{\cdot}$ are the average trace operator and jump operator, respectively, defined as
\begin{align*}
\av{\phi} \big|_f &:= \frac{1}{|\mathcal{T}_f|}\sum_{e\in\mathcal{T}_f} \phi|_{\partial e\cap f}, & \ju{\vu} \big|_f &:= \sum_{e\in\mathcal{T}_f} (\vun \vu)|_{\partial e\cap f}, 
\end{align*}
for all faces $f\in\mathcal{F}$, where $\mathcal{T}_f$ denotes the set of elements adjacent to face $f$, and $\vun|_{\partial e}$ denotes the outward pointing normal unit vector of element $e$. The bilinear form $a^{(C)}_h$ is the same as the original elliptic operator $a$ and is the part that remains when both input functions are continuous. The bilinear form $a^{(D)}_h$ can be interpreted as the additional part that results from partial integration of the elliptic operator $a$ when the first input function is discontinuous. Finally, the bilinear form $a^{(IP)}_h$ is the part that contains the interior penalty function needed to ensure stability of the scheme. 

The penalty term can have a significant impact on the performance of the SIPDG method, since a larger penalty term results in a more restrictive bound on the time step size, but also because it can have a significant effect on the accuracy, as we will show in Section \ref{sec:results}. Several lower bounds for the penalty term are based on the trace inequality of \cite{warburton03}, including \cite{shahbazi05, epshteyn07, mulder14}, among which we found the bound in \cite{mulder14} to be the sharpest. Recently, a sharper penalty term bound was presented in \cite{geevers17}, which is based on a more involved trace inequality. In this paper we will consider both the penalty term of \cite{geevers17}, given by (\ref{eq:penTerm_a}), and the one of \cite{mulder14}, given by (\ref{eq:penTerm_b}):
\begin{subequations}
\begin{align}
\alpha_h|_{\partial e\cap f} &:= \frac{\nu_h|_{\partial e\cap f}}{|\mathcal{T}_f|}\sup_{\substack{\vu\in \mathcal{P}^p(e)^m \\ \;C:\nabla\vu\neq\tzero }}  \frac{\displaystyle\int_{\partial e} (\vun\cdot C:\nabla\vu)\cdot \nu_h^{-1}\ten{c}_{\vun}^{-1}\cdot(\vun\cdot C:\nabla\vu) \;ds}{ \displaystyle\int_e (\nabla\vu)^t : C : \nabla\vu \;d\vx},  \label{eq:penTerm_a} \\
\alpha_h|_{\partial e\cap f} &:= \frac{p(p+2)}{\min_{e\in\mathcal{T}_f} d_{e}}, \label{eq:penTerm_b}
\end{align}
\label{eq:penTerm}%
\end{subequations}
for all $e\in\mathcal{T}_h$, $f\subset\partial e$, where $p$ denotes the degree of the polynomial basis functions, $\mathcal{P}^p(e)$ denotes the space of polynomial functions of degree $p$ or less in element $e$, $\nu_h|_{\partial e\cap f}:=|f|/|e|$ is a scaling function of order $h^{-1}$, with $|e|,|f|$ the volume of $e$ and area of $f$, respectively, $\ten{c}_{\vun}^{-1}$ denotes the (pseudo)-inverse of the second-order tensor $\ten{c}_{\vun}:=\vun\cdot C\cdot\vun$, where $\vun$ is the outward pointing normal unit vector, and $d_{e}$ denotes the diameter of the inscribed sphere of $e$. Although the first version requires more preprocessing time, it allows for an approximately $1.5$ times larger time step \cite{geevers17}. 

We will refer to the SIPDG method with $p=1,2,3$ using the penalty term as defined by (\ref{eq:penTerm_a}) as DG1a, DG2a, and DG3a, respectively, and to the same methods using the penalty term as defined by (\ref{eq:penTerm_b}) as DG1b, DG2b, and DG3b.

\subsection{The Lax--Wendroff Time Integration Scheme}
To solve the resulting set of ODE's (\ref{eq:ODE}) in time, we use the Lax--Wendroff method \cite{lax64, deBasabe10}, which is based on Taylor expansions in time and substitutes the time derivatives by matrix-vector operators using the original equations (\ref{eq:ODE}). For the second-order formulation, the resulting scheme is also known as Dablain's scheme \cite{dablain86}. The advantage of this scheme is that it is time-reversible, energy-conservative, and only requires $K$ stages for a $2K$-order of accuracy.

To introduce the scheme, let $\Delta t>0$ denote the time step size, and let $\vvU_h(t_i)$ denote the approximation of $\vvu_h$ at time $t_i:=i\Delta t$ for $i=0,\dots, N_T$ with $N_T$ the total number of time steps. The order-$2K$ Lax--Wendroff method can be written as
\begin{align}
\vvU_h(t_{i+1}) &= -\vvU_h(t_{i-1}) + 2\sum_{k=0}^K\frac{1}{(2k)!} \Delta t^{2k}(\partial_t^{2k}\vvct{U}_h)(t_i), &&i=1,\dots, N_T-1,
\label{eq:LW}
\end{align}
with $\vvU_h(t_0)=\vvU_h(0):=\vvct{u_0}$ and $\vvU_h(t_1):=\sum_{k=0}^{2K+1} \frac{1}{k!} \Delta t^k(\partial_t^{k}\vvct{U}_h)(0)$, and where $(\partial_t^{k}\vvct{U}_h)(t_i)$ is recursively defined by
\begin{align*}
(\partial_t^{k}\vvct{U}_h)(0) &:= \begin{cases}
\vvct{u_0} & k=0, \\
\vvct{v_0} & k=1, \\
-M^{-1}A(\partial_t^{k-2}\vvU_h)(0)+\partial_t^{k-2}\vvct{f}(0) & k\geq 2,  
\end{cases}
\end{align*}
and
\begin{align*}
(\partial_t^{k}\vvct{U}_h)(t_i) &:= \begin{cases}
\vvU_h(t_i) & k=0, \\
-M^{-1}A(\partial_t^{k-2}\vvU_h)(t_i)+\partial_t^{k-2}\vvct{f}(t_i) & k=2,4,6,\dots,2K, 
\end{cases}
\end{align*}
for $i\geq 1$, with $\vvct{f}:=M^{-1}\vvct{f}^*$. In case $K=1$, this scheme reduces to the standard leap-frog or central difference scheme. When there is no source term, (\ref{eq:LW}) simplifies to
\begin{align}
\vvU_h(t_{i+1}) &= -\vvU_h(t_{i-1}) + 2\sum_{k=0}^K \frac{1}{(2k)!} \Delta t^{2k} (-M^{-1}A)^k\vvU_h(t_i),
\label{eq:LW0}
\end{align}
for $i=1,\dots,N_T-1$.

For the dispersion analysis, we will choose $K$ equal to the polynomial degree $p$ of the spatial discretization, since this will result in a $2p$-order convergence rate of the dispersion error as shown in Section \ref{sec:results}.

\section{Dispersion Analysis}
\label{sec:dispAn}

A common measure for the quality of a numerical method for wave propagation modelling is the amount of numerical dispersion and dissipation. Numerical dispersion refers in this context to the discrepancy between the numerical and physical wave propagation speed and numerical dissipation is the loss of energy in the numerical scheme. Since the schemes that we consider are all energy-conservative, they do not suffer from numerical dissipation. However, when projecting a physical wave onto the discrete space, this results in a superposition of a well-matching numerical wave and several numerical waves that have a completely different shape and frequency. We compute the number of these non-matching or spurious waves and refer to it as the eigenvector error, since it is related to the accuracy of the eigenvectors of $M^{-1}A$, while the dispersion error is related to the accuracy of the eigenvalues of $M^{-1}A$.

\begin{figure}[h]
\centering
\begin{subfigure}[b]{0.45\textwidth}
  \includegraphics[width=\textwidth]{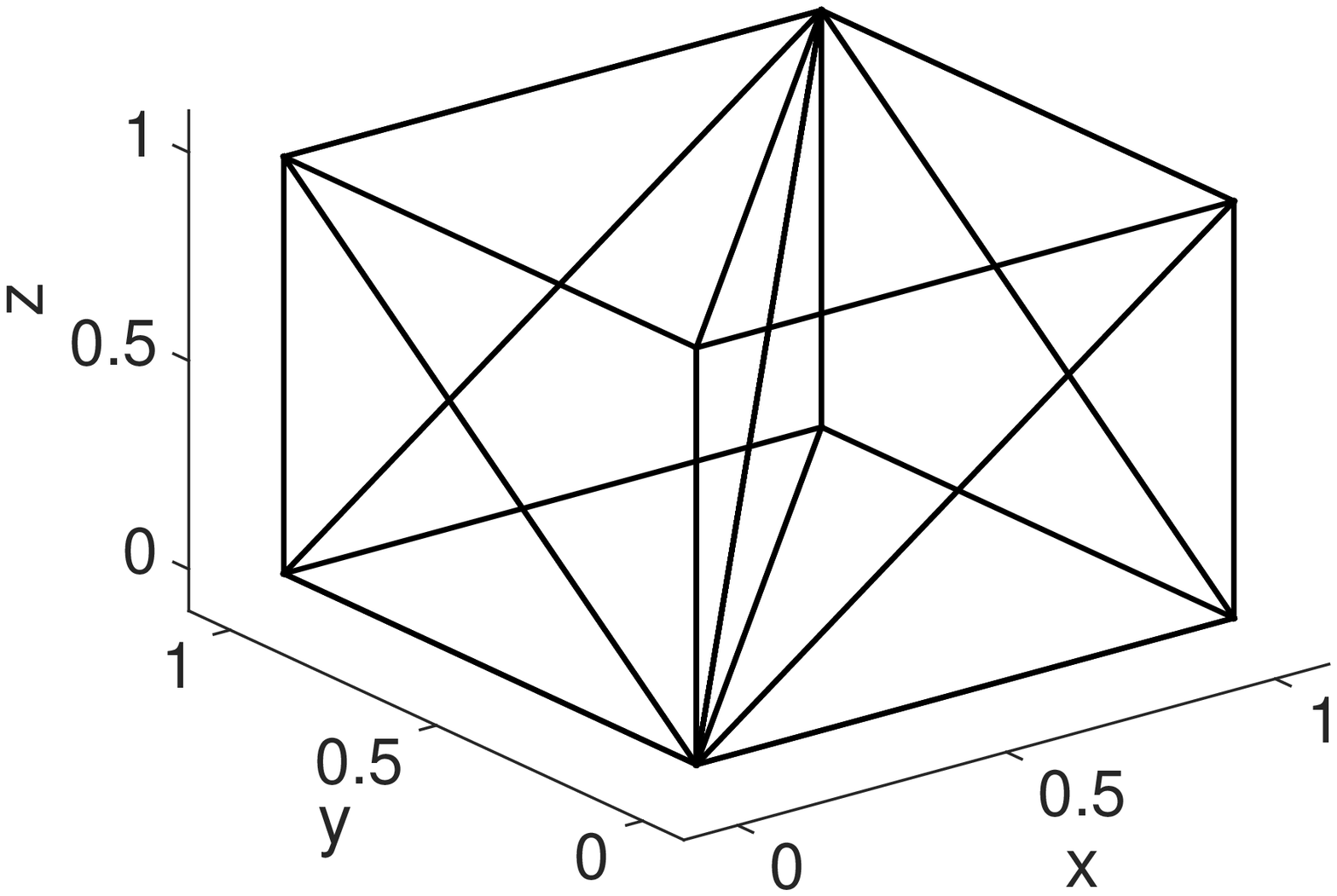} 
\end{subfigure} \,\,
\begin{subfigure}[b]{0.45\textwidth}
  \includegraphics[width=\textwidth]{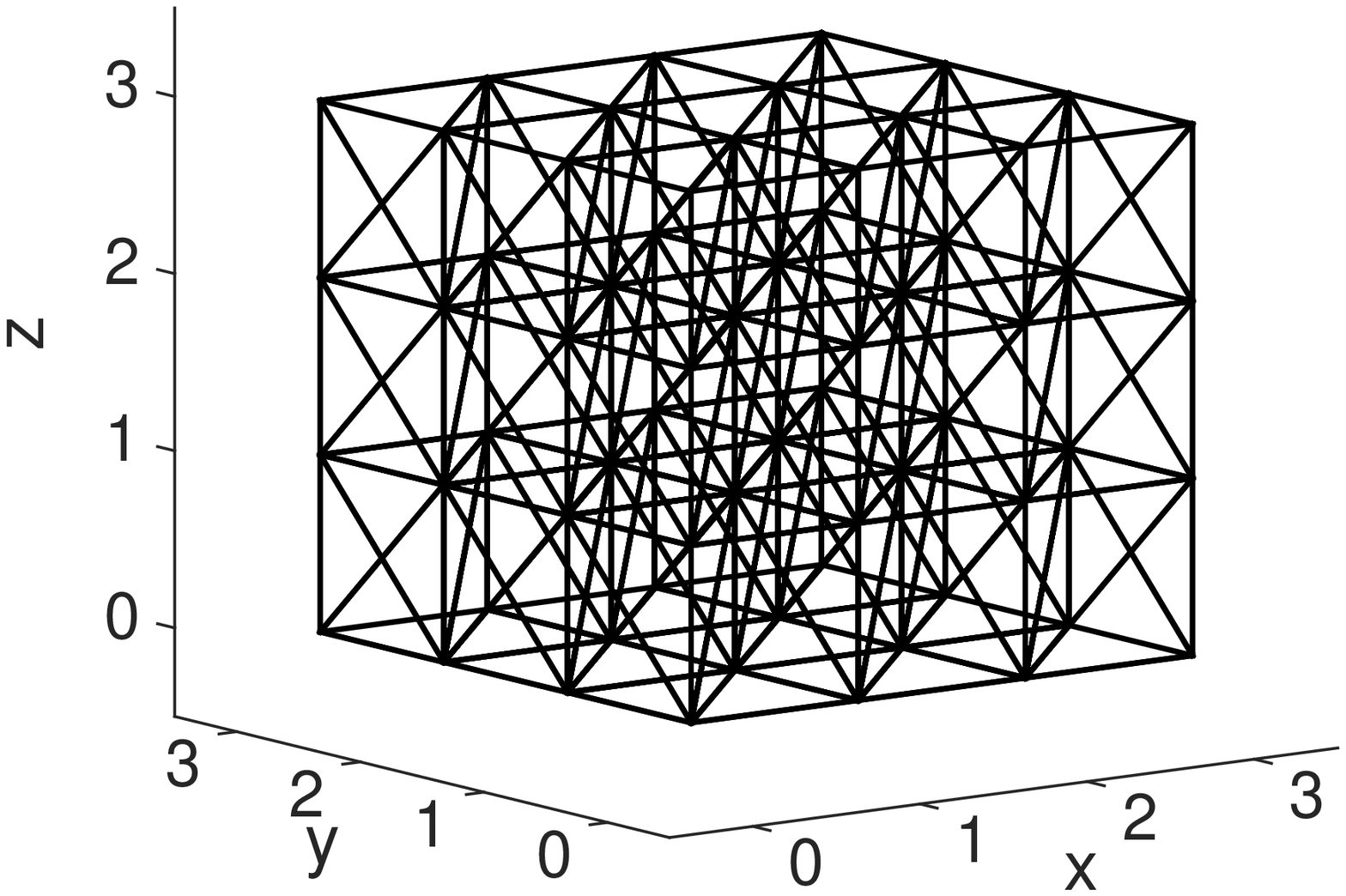} 
\end{subfigure}
\caption{Unit cell subdivided into tetrahedra (left), and periodic mesh made from $3\times 3\times 3$ copies of this unit cell (right).}
\label{fig:mesh}
\end{figure}

We analyse the dispersion and eigenvector error using standard Fourier analysis, which is also known in this context as plane wave analysis. The main idea of this analysis is to compare physical plane waves with numerical plane waves on a homogeneous periodic domain, free from external forces, using a periodic mesh. To obtain a periodic tetrahedral mesh we subdivide a small cell into tetrahedra and repeat this pattern to fill the entire domain as illustrated in Figure \ref{fig:mesh}. By using Fourier modes, we can then efficiently compute the numerical plane waves and their dispersion properties by solving eigenvalue problems on only a single cell.

Our analysis is similar to \cite{deBasabe07}, but with the following extensions:
\begin{itemize}
  \item We extend the analysis to parallellepiped cells, since this allows for a more regular tetrahedral mesh.
  \item We also compute the number of spurious modes that appear in the projection of the physical wave.
  \item In the three-dimensional elastic case, there are two distinct secondary or shear waves with the same wave vector. To compute the dispersion and eigenvector error in this case, we consider the two best matching numerical plane waves.
\end{itemize}
 
We explain the dispersion analysis in more detail in the following subsections. First, we show how we can derive an analytical expression for the numerical plane waves using Fourier modes. After that, we show how we use this to compute the numerical dispersion and eigenvector error. In the last subsection we explain how we estimate the computational cost for each method.

\subsection{Analytic Expression for the Numerical Plane Waves}
We first consider a periodic cubic domain of the form $\Omega:=[0,N)^3$, with $N$ a positive integer, and later extend the results to parallelepiped domains which allow for more regular tetrahedral meshes. The physical plane wave has the following form:
\begin{align}
\vu(\vx,t) &= \vct{a}e^{\im (\vkappa\cdot\vx -\omega t)}, &&\vx\in\Omega, t\in[0,T],
\label{eq:physWave}
\end{align}
where $\im:=\sqrt{-1}$ is the imaginary number, $\vkappa\in\mathbb{R}^3$ is the wave vector, $\omega\in\mathbb{R}$ is the angular velocity, and $\vct{a}\in\mathbb{R}^m$ is the amplitude vector. The wave vector must be of the form $\vkappa=\vkappa_{\vct{z}}=\frac{2\pi}{N}\vct{z}$, with $\vct{z}\in\mathbb{Z}_N^3$, in order to satisfy the periodic boundary conditions. 

The numerical plane wave can be written in a similar form when using a periodic mesh. To obtain a periodic tetrahedral mesh, we subdivide the unit cell $\Omega_0:=[0,1)^3$ into tetrahedra and repeat this pattern $N\times N\times N$ times to fill the entire domain as illustrated in Figure \ref{fig:mesh}. We equip the mesh with a translation-invariant set of basis functions where each basis function has minimal support. In case of mass-lumping we use nodal basis functions and in case of DG we use basis functions that have support on only a single element. The numerical plane wave $\vvU_h$ of the fully discrete scheme then has the form
\begin{align}
\vvU_{h}(\Omega_{\vct{k}},t_i) &= \vvU_{h,\Omega_0}e^{\im (\vkappa\cdot\vx_{\vct{k}} -\omega_h t)}, &&i=0,\dots,N_T, \vct{k}\in\mathbb{Z}_N^3.
\label{eq:numWave}
\end{align}
Here, $\vvU_h(\Omega_\vct{k},t_i)$ denotes the coefficients of the basis functions corresponding to cell $\Omega_{\vct{k}}:=\vct{k}+\Omega_0$ at time $t_i$. In case of mass-lumping, these basis functions are the nodal basis functions corresponding to the nodes on $\Omega_{\vct{k}}=\vct{k}+[0,1)^3$, while in case of DG, these are the basis functions that have support on one of the tetrahedra in $\Omega_{\vct{k}}$. The vector $\vvU_{h,\Omega_0}\in\mathbb{R}^{n_0}$ denotes the basis function coefficients corresponding to cell $\Omega_0$ at time $0$ and $\vx_{\vct{k}}=\vct{k}$ are the coordinates of the front-left-bottom vertex of cell $\Omega_{\vct{k}}$. 

To show that this is indeed a numerical plane wave, let $M^{(\Omega_\vct{k},\Omega_\vct{m})}, A^{(\Omega_\vct{k},\Omega_\vct{m})} \in\mathbb{R}^{n_{0}\times n_{0}}$, for $\vct{k},\vct{m}\in\mathbb{Z}_N^3$, be submatrices of $M$ and $A$, respectively, defined as follows:
\begin{align*}
M^{(\Omega_\vct{k},\Omega_\vct{m})}_{ij} &:= \left(\rho \vw^{(\Omega_\vct{k},i)}, \vw^{(\Omega_\vct{m},j)}\right)_h, &&i,j=1,\dots,n_{0},  \\
A^{(\Omega_\vct{k},\Omega_\vct{m})}_{ij} &:= a_h\left( \vw^{(\Omega_\vct{k},i)}, \vw^{(\Omega_\vct{m},j)}\right), &&i,j=1,\dots,n_{0},
\end{align*}
where $\{\vw^{(\Omega_{\vct{k}},i)}\}_{i=0}^{n_0}$ denote the basis functions corresponding to cell $\Omega_{\vct{k}}$ and where $a_h = a_h^{(DG)}$ and $(\cdot,\cdot)_h = (\cdot,\cdot)$ in case of the DG method and $a_h=a$ and $(\cdot,\cdot)_h=(\cdot,\cdot)_h^{(L)}$ in case of the mass-lumped method. 

Since the basis functions are translation invariant, the submatrices $M^{(\Omega_\vct{k},\Omega_{\vct{k}+\Delta\vct{k}})}$ and  $A^{(\Omega_\vct{k},\Omega_{\vct{k}+\Delta\vct{k}})}$ are the same for any $\vct{k}\in\mathbb{Z}_N^3$ with $\Delta\vct{k}\in\mathbb{Z}_N^3$ fixed. Furthermore, the submatrices $M^{(\Omega_\vct{k},\Omega_\vct{m})}$ are only non-zero when $\vct{k}=\vct{m}$, since the mass matrix is diagonal in case of mass-lumping and block-diagonal, with each block corresponding to an element, in case of DG. The submatrices $A^{(\Omega_\vct{k},\Omega_{\vct{k}+\Delta\vct{k}})}$ are only non-zero when $\Delta\vct{k}\in\{-1,0,1\}^3$, since the nodal basis functions for mass-lumping and the local basis functions for DG do not interact when they are two or more cells apart. This implies that we only need to consider the submatrices $M^{(\Omega_0)}:=M^{(\Omega_0,\Omega_0)}$ and $A^{(\Omega_0,\Omega_{\Delta\vct{k}})}$ for $\Delta\vct{k}=\{-1,0,1\}^3$.

Now let $\vkappa=\vkappa_\vct{z}:=\frac{2\pi}{N}\vct{z}$, for some $\vct{z}\in\mathbb{Z}_N^3$, and let $\vvct{U}_{h,0}\in\mathbb{R}^{N^3\times n_0}$ be the numerical wave at time $t=0$:
\begin{align}
\vvct{U}_{h,0}(\Omega_{\vct{k}}) &:= \vvct{U}_{h,\Omega_0}e^{\im (\vkappa\cdot\vx_{\vct{k}})}, && \vct{k}\in\mathbb{Z}_N^3.
\label{eq:eigVec}
\end{align}
Then $M^{-1}A\vvct{U}_{h,0}$ satisfies 
\begin{align*}
\left(M^{-1}A\vvct{U}_{h,0}\right)(\Omega_{\vct{k}}) &= M^{(\Omega_0)}_{inv} \left(\sum_{\Delta\vct{k}\in\{-1,0,1\}^3} A^{(\Omega_0,\Omega_{\Delta\vct{k}})}\vvct{U}_{h,0}(\Omega_{\vct{k}+\Delta\vct{k}})\right) \\
&=M^{(\Omega_0)}_{inv} \left(\sum_{\Delta\vct{k}\in\{-1,0,1\}^3} e^{\im(\vkappa\cdot\vx_{\Delta\vct{k}})}A^{(\Omega_0,\Omega_{\Delta\vct{k}})}\right) \vvct{U}_{h,\Omega_0}e^{\im (\vkappa\cdot\vx_{\vct{k}})} \\
&= M^{(\Omega_0)}_{inv}A^{(\vkappa)}\vvct{U}_{h,\Omega_0}e^{\im (\vkappa\cdot\vx_{\vct{k}})},
\end{align*}
for all $\vct{k}\in\mathbb{Z}_N^3$, with $M^{(\Omega_0)}_{inv}$ the inverse of $M^{(\Omega_0)}$ and
\begin{align*}
A^{(\vkappa)} &:= \sum_{\Delta\vct{k}\in\{-1,0,1\}^3} e^{\im(\vkappa\cdot\vx_{\Delta\vct{k}})}A^{(\Omega_0,\Omega_{\Delta\vct{k}})}.
\end{align*}
This implies that if $(s_h,\vvct{U}_{h,\Omega_0})$ is an eigenpair of $S^{(\vkappa)}:=M^{(\Omega_0)}_{inv}A^{(\vkappa)}$, then $(s_h,\vvct{U}_{h,0})$ is an eigenpair of $M^{-1}A$. In other words, we can obtain eigenpairs of $M^{-1}A$ by computing the eigenpairs of a small matrix $S^{(\vkappa)}\in\mathbb{R}^{n_0\times n_0}$. Note that since $M^{(\Omega_0)}$ is symmetric positive definite, and $A^{(\vkappa)}$ is Hermitian, $S^{(\vkappa)}$ has $n_0$ distinct eigenpairs. Since there are $N^3$ choices for $\vct{z}\in\mathbb{Z}_N^3$ and $S^{(\vkappa_\vct{z})}$ has $n_0$ eigenpairs, we can obtain all of the $N^3\times n_0$ eigenpairs of $M^{-1}A$ in this way.

Now consider the numerical plane wave in (\ref{eq:numWave}) with $(s_h,\vvU_{h,\Omega_0})$ an eigenpair of $S^{(\vkappa)}$, so with $(s_h,\vvct{U}_{h,0})$ an eigenpair of $M^{-1}A$. We can rewrite $\vvU_h$ as $\vvU_{h}(t)=\vvU_{h,0}e^{-\im(\omega t)}$. If we then substitute this wave into (\ref{eq:LW0}) we obtain
\begin{align*}
\cos(\Delta t\omega_h)\vvU_h(t_i) &= \sum_{k=0}^K\frac{1}{(2k)!}(-\Delta t^2s_h)^{k}\vvU_h(t_i), &&i=1,\dots,n_T-1.
\end{align*}
From this, it follows that $\vvU_h$ in (\ref{eq:numWave}) is a discrete plane wave if $(s_h,\vvU_{h,\Omega_0})$ is an eigenpair of $S^{(\vkappa)}$ and if $\omega_h$ satisfies $\cos(\Delta t\omega_h) =\sum_{k=0}^K\frac{1}{(2k)!}(-\Delta t^{2}s_h)^{k}$, so if
\begin{align}
\omega_h &=\pm\frac{1}{\Delta t}\arccos\left(\sum_{k=0}^K\frac{1}{(2k)!}(-\Delta t^{2}s_h)^{k}\right).
\label{eq:omgh}
\end{align}

It remains to determine the time step size $\Delta t$. In the appendix we show that the numerical scheme is stable, if 
\begin{align}
\Delta t\leq \sqrt{c_K/\sigma_{max}(M^{-1}A)},
\label{eq:CFL}
\end{align}
where  $\sigma_{max}(M^{-1}A)$ denotes the spectral radius of $M^{-1}A$ and $c_K$ is a constant, given by
\begin{align}
c_K := \inf\left\{x\geq 0 \;|\; \left|\sum_{k=0}^K \frac{1}{(2k)!}(-x)^k\right| > 1 \right\}.
\label{eq:cK}
\end{align}


To obtain a bound on the spectral radius, recall that we can write every eigenpair of $M^{-1}A$ in the form of $(s_h,\vvct{U}_{h,0})$, with $\vvct{U}_{h,0}$ given in (\ref{eq:eigVec}) and with $(s_h,\vvU_{h,\Omega_0})$ an eigenpair of $S^{(\vkappa_{\vct{z}})}$ for some $\vct{z}\in\mathbb{Z}_N^3$. This implies that $\sigma_{max}(M^{-1}A)$ is equal to $\sup_{\vct{z}\in\mathbb{Z}_N^3}\sigma_{max}(S^{(\vkappa_\vct{z})})$. We can therefore bound  $\sigma_{max}(M^{-1}A)$ as follows:
\begin{align}
\sigma_{max}(M^{-1}A) = \sup_{\vct{z}\in\mathbb{Z}_N^3}\sigma_{max}(S^{(\vkappa_\vct{z})}) \leq \sup_{\vkappa\in\mathcal{K}_0} \sigma_{max}(S^{(\vkappa)}) =:s_{h,max},
\label{eq:spectralBound}
\end{align}
with $\mathcal{K}_0:=[0,2\pi)^3\supset\{\vkappa_\vct{z}\}_{\vct{z}\in\mathbb{Z}_N^3}$ the space of all distinct wave vectors $\vkappa$. 

The constants $c_K$ can be computed numerically. For example, $c_K=4,12,7.57$ for $K=1,2,3$, respectively. For higher values of $K$, see, for example, \cite{mulder13}, where his $\sigma_t$ satisfies $c_K=2\sigma_t$.

We can extend the results of this section to parallelepiped cells by applying a linear transformation $\vct{x}\rightarrow\ten{T}\cdot\vct{x}$, with $\ten{T}\in\mathbb{R}^{3\times 3}$ a second-order tensor. The parallelepiped domain is then given by $\Omega=\ten{T}\cdot(0,N)^3$ and the cells are given by $\Omega_0=\ten{T}\cdot[0,1)^3$ and $\Omega_{\vct{k}}=\vct{x}_{\vct{k}} + \Omega_0$, with $\vct{x}_{\vct{k}}=\ten{T}\cdot \vct{k}$ the front-left-bottom vertex. The wave vectors $\vkappa_\vct{z}$ are of the form $\vkappa_{\vct{z}}=\frac{2\pi}{N}(\ten{T}^{-t}\cdot\vct{z})$ and the wave vector space $\mathcal{K}_0$ is given by $\mathcal{K}_0:=\ten{T}^{-t}\cdot[0,2\pi)^3$, with $\ten{T}^{-t}$ the transposed inverse of $\ten{T}$.

\subsection{Computing the Dispersion and Eigenvector Error }
To explain how we compute the dispersion error, we first consider the acoustic wave equation. Let $\vkappa$ be a given wave vector and let $\vu^{(\vkappa)}$ be the acoustic plane wave given by $\vu^{(\vkappa)}(\vx,t)=e^{\im(\vkappa\cdot\vx-\omega t)}$. The angular velocity is given by $\omega=\pm c|\vkappa|$ with $c$ the acoustic wave propagation speed. We compare this plane wave with the numerical plane waves.

To do this, we use the results from the previous subsection. There, we showed that for any eigenpair $(s_h,\vvU_{h,\Omega_0})$ of $S^{(\vkappa)}$ we can obtain a numerical plane wave in the form of (\ref{eq:numWave}) with angular velocity $\pm\omega_h$  given by (\ref{eq:omgh}). Since $S^{(\vkappa)}$ has $n_0$ eigenpairs, this means we can obtain $n_0$ discrete plane waves $\{\vvU^{(\vkappa,i)}_h\}_{i=1}^{n_0}$, with angular velocities $\{\pm\omega^{(\vkappa,i)}_{h}\}_{i=1}^{n_0}$, for a given wave vector $\vkappa$. The corresponding wave propagation speeds $\{c^{(\vkappa,i)}_{h}\}_{i=1}^{n_0}$ can be computed by $c^{(\vkappa,i)}_{h}=|\omega^{(\vkappa,i)}_{h}|/|\vkappa|$ and we can order the numerical plane waves such that 
\begin{align*}
|c-c^{(\vkappa,1)}_{h}|\leq|c-c^{(\vkappa,2)}_{h}|\leq\dots.
\end{align*}
We consider $\vvU^{(\vkappa,1)}_h$ to be the matching numerical plane wave and $\vvU^{(\vkappa,i)}_h$, with $i>1$, to be spurious modes. We then define the dispersion error as follows
\begin{align*}
e_{disp}(\vkappa) &= \frac{|c-c_{h}^{(\vkappa,1)}|}{c}.
\end{align*}

The complete procedure for computing $e_{disp}(\vkappa)$ in the acoustic case is given by
\begin{enumerate}[a.]
  \item Compute all eigenpairs $(s_h^{(\vkappa,i)}, \vvU_{h,\Omega_0}^{(\vkappa,i)})$ of $S^{(\vkappa)}:=M^{(\Omega_0)}_{inv}A^{(\vkappa)}$.
  \item Compute the angular velocities $\omega_{h}^{(\vkappa,i)}=\frac{1}{\Delta t}\arccos\left(\sum_{k=0}^K \frac{1}{(2k)!}(-\Delta t^2s_{h}^{(\vkappa,i)})^k\right)$.
  \item Compute the wave propagation speeds $c_h^{(\vkappa,i)}=\omega_{h}^{(\vkappa,i)}/|\vkappa|$ and order everything such that $|c-c^{(\vkappa,1)}_{h}|\leq|c-c^{(\vkappa,2)}_{h}|\leq\dots$.
  \item Compute $e_{disp}(\vkappa)=|c-c_h^{(\vkappa,1)}|/c$.
\end{enumerate}

Now let $\vu_0^{(\vkappa)}(\vx):=\vu^{(\vkappa)}(\vx,0)=e^{\im(\vkappa\cdot\vx)}$ be the acoustic plane wave at $t=0$. Also, let $\vvu^{(\vkappa)}_0$ be the projection of $\vu_0^{(\vkappa)}$ onto the numerical space, and let $\vvU_{h,0}^{(\vkappa,i)}$ be the discrete plane wave at $t=0$. In the ideal case, $\vvu_0^{(\vkappa)}$ is equal to $\vvU_{h,0}^{(\vkappa,1)}$ up to a constant. In most cases, however, the projection $\vvu_0^{(\vkappa)}$ is a superposition of a well-matching plane wave $\vvU_{h,0}^{(\vkappa,1)}$ and several other plane waves $\vvU_{h,0}^{(\vkappa,i)}$, for $i>1$, that may have a completely different shape and velocity. We can compute the number of these spurious waves by computing the projection error. 

To do this, we let $\vvU^{(\vkappa)}_0\in\mathrm{span}\{\vvU_{h,0}^{(\vkappa,1)}\}$ denote the projection of $\vvu^{(\vkappa)}_0$ onto $\mathrm{span}\{\vvU_{h,0}^{(\vkappa,1)}\}$, such that $(\vvU^{(\vkappa)}_0,\vvU_{h,0}^{(\vkappa,1)})_M=(\vvu^{(\vkappa)}_0,\vvU_{h,0}^{(\vkappa,1)})_M$, with $(\vvu,\vvv)_M:=\vvu^tM\vvv$. We then define the projection error as
\begin{align*}
e_{vec}(\vkappa) := \frac{\|\vvu^{(\vkappa)}_0-\vvU_{0}^{(\vkappa)}\|_M}{\|\vvu_{0}^{(\vkappa)}\|_M},
\end{align*}
where $\|\vvu\|_M:=\sqrt{\vvu^tM\vvu}$. We refer to this as the eigenvector error, since it is related to the accuracy of $\vvU_{h,0}^{(\vkappa,1)}$, which is an eigenvector of $M^{-1}A$ \cite{mulder99}. 

Since the physical plane wave, the mesh, and the set of basis functions are all translation invariant, we can efficiently compute this error by only considering $\vu^{(\vkappa)}_{\Omega_0}$, the part of $\vu^{(\vkappa)}_0$ restricted to cell $\Omega_0$. We define $\vvu^{(\vkappa)}_{\Omega_0}$ to be the projection of $\vu^{(\vkappa)}_{\Omega_0}$ onto the discrete space restricted to $\Omega_0$ and define $\vvU^{(\vkappa)}_{\Omega_0}\in\mathrm{span}\{\vvU_{h,\Omega_0}^{(\vkappa,1)}\}$ the projection of $\vvu^{(\vkappa)}_{\Omega_0}$ onto $\mathrm{span}\{\vvU_{h,\Omega_0}^{(\vkappa,1)}\}$ such that $(\vvU^{(\vkappa)}_{\Omega_0},\vvU_{h,\Omega_0}^{(\vkappa,1)})_{M_0}=(\vvu^{(\vkappa)}_{\Omega_0},\vvU_{h,\Omega_0}^{(\vkappa,1)})_{M_0}$, with $M_0:=M^{(\Omega_0)}$. We can then compute $e_{vec}(\vkappa)$ by
\begin{align*}
e_{vec}(\vkappa)=\frac{\|\vvu^{(\vkappa)}_{\Omega_0}-\vvU_{\Omega_0}^{(\vkappa)}\|_{M_0}}{\|\vvu_{\Omega_0}^{(\vkappa)}\|_{M_0}}.
\end{align*}

The complete procedure for computing $e_{vec}(\vkappa)$ in the acoustic case is given by
\begin{enumerate}[A.]
  \item Compute all eigenpairs $(s_h^{(\vkappa,i)}, \vvU_{h,\Omega_0}^{(\vkappa,i)})$ of $S^{(\vkappa)}:=M^{(\Omega_0)}_{inv}A^{(\vkappa)}$.
  \item Compute the angular velocities $\omega_{h}^{(\vkappa,i)}=\frac{1}{\Delta t}\arccos\left(\sum_{k=0}^K \frac{1}{(2k)!}(-\Delta t^2s_{h}^{(\vkappa,i)})^k\right)$.
  \item Compute the wave propagation speeds $c_h^{(\vkappa,i)}=\omega_{h}^{(\vkappa,i)}/|\vkappa|$ and order everything such that $|c-c^{(\vkappa,1)}_{h}|\leq|c-c^{(\vkappa,2)}_{h}|\leq\dots$.
  \item Compute $\vvu^{(\vkappa)}_{\Omega_0}$: the projection of $\vu^{(\vkappa)}_{\Omega_0}$ onto the discrete space of cell $\Omega_0$.
  \item Compute $\vvU^{(\vkappa)}_{\Omega_0}$: the projection of $\vvu^{(\vkappa)}_{\Omega_0}$ onto $\mathrm{span}\{\vvU_{h,\Omega_0}^{(\vkappa,1)}\}$
  \item Compute $e_{vec}(\vkappa)={\|\vvu^{(\vkappa)}_{\Omega_0}-\vvU_{\Omega_0}^{(\vkappa)}\|_{M_0}}/{\|\vvu_{\Omega_0}^{(\vkappa)}\|_{M_0}}$.
\end{enumerate}

For the isotropic elastic case, the procedure is very similar. Let $\vkappa$ be the wave vector and let $\vu^{(\vkappa)}$ denote the elastic plane wave of the form $\vu^{(\vkappa)}(\vx,t)=\vct{a}e^{\im(\vkappa\cdot\vx-\omega t)}$, with $\vct{a}$ the amplitude vector, $\omega=\pm c|\vkappa|$ the angular velocity, and $c$ the elastic wave propagation speed. In the elastic isotropic case, we have to distinguish between longitudinal or primary waves, where $\vct{a}$ is parallel with $\vkappa$ and the propagation speed is $c=c_P=\sqrt{(\lambda+2\mu)/\rho}$, and transversal, shear or secondary waves, where $\vct{a}$ is perpendicular to $\vkappa$ and the propagation speed is $c=c_S=\sqrt{\mu/\rho}$. 

For the analysis, we will only consider the secondary wave, since the wavelength $\lambda=2\pi/|\kappa|=2\pi c/\omega$ of this wave is shorter and therefore governs the required mesh resolution. In 3D, there are two linear independent amplitude vectors, $\vct{a}^{(\vkappa,1)}$ and $\vct{a}^{(\vkappa,2)}$, that are perpendicular to $\vkappa$ and we will refer to the corresponding secondary plane waves as $\vvu^{(\vkappa,1)}$ and $\vvu^{(\vkappa,2)}$. We will compare these physical plane waves with the numerical plane waves in a similar way as for the acoustic case.

Since, for a given $\vkappa$ and $\omega=\pm c_S|\vkappa|$, there are two linearly independent secondary waves, we compare the secondary wave velocity $c=c_S$ with the wave propagation speed of the two best matching numerical plane waves. In particular, we define the dispersion error as
\begin{align*}
e_{disp}(\vkappa) &= \frac{|c-c_{h}^{(\vkappa,2)}|}{c}.
\end{align*}

The procedures for computing this error is the same as for the acoustic case, with step d replaced by
\begin{enumerate}
  \item[d*.] Compute $e_{disp}(\vkappa) = {|c-c_{h}^{(\vkappa,2)}|}/{c}.$
\end{enumerate}

The eigenvector is now computed by
\begin{align*}
e_{vec}(\vkappa) = \sup_{\vvu^{(\vkappa)}_{\Omega_0}\in\mathrm{span}\{\vvu^{(\vkappa,1)}_{\Omega_0},\vvu^{(\vkappa,2)}_{\Omega_0}\}} \frac{\|\vvu^{(\vkappa)}_{\Omega_0}-\vvU_{\Omega_0}^{(\vkappa)}\|_{M_0}}{\|\vvu_{\Omega_0}^{(\vkappa)}\|_{M_0}},
\end{align*}
where $\vvu^{(\vkappa,i)}_{\Omega_0}$ is the projection of $\vu^{(\vkappa,i)}_{\Omega_0}$ onto the discrete space of cell $\Omega_0$, and $\vvU_{\Omega_0}^{(\vkappa)}\in\mathrm{span}\{\vvU_{h,\Omega_0}^{(\vkappa,1)}, \vvU_{h,\Omega_0}^{(\vkappa,2)}\}$ is the projection of $\vvu_{\Omega_0}^{(\vkappa)}\in\mathrm{span}\{\vvu^{(\vkappa,1)}_{\Omega_0},\vvu^{(\vkappa,2)}_{\Omega_0}\}$ onto $\mathrm{span}\{\vvU_{h,\Omega_0}^{(\vkappa,1)}, \vvU_{h,\Omega_0}^{(\vkappa,2)}\}$. In other words, we compute the worst possible projection error for a linear combination of $\vvu^{(\vkappa,1)}_{\Omega_0}$ and $\vvu^{(\vkappa,2)}_{\Omega_0}$ projected onto the span of the two best-matching numerical plane waves $\vvU_{h,\Omega_0}^{(\vkappa,1)}$ and $\vvU_{h,\Omega_0}^{(\vkappa,2)}$. We can efficiently compute this by
\begin{align*}
e_{vec}(\vkappa) = \sqrt{\sigma_{max}(B^{-1}R)},
\end{align*}
where $\sigma_{max}(B^{-1}R)$ denotes the largest eigenvalue of $B^{-1}R$ and $B,R\in\mathbb{R}^{2\times 2}$ are matrices given by $B_{ij}:=(\vvu_{\Omega_0}^{(\vkappa,i)},\vvu_{\Omega_0}^{(\vkappa,j)})_{M_0}$ and $R_{ij}:=(\vvu^{(\vkappa,i)}_{\Omega_0}-\vvU_{\Omega_0}^{(\vkappa,i)},\vvu^{(\vkappa,j)}_{\Omega_0}-\vvU_{\Omega_0}^{(\vkappa,j)})_{M_0}$, with $\vvU_{\Omega_0}^{(\vkappa,i)}$ the projection of $\vvu_{\Omega_0}^{(\vkappa,i)}$ onto $\mathrm{span}\{\vvU_{h,\Omega_0}^{(\vkappa,1)}, \vvU_{h,\Omega_0}^{(\vkappa,2)}\}$.

The procedure for computing $e_{vec}(\vkappa) $ is the same as for the acoustic case, with steps D-F replaced by
\begin{enumerate}
  \item[D*.] Compute $\vvu^{(\vkappa,i)}_{\Omega_0}$: the projection of $\vu^{(\vkappa,i)}_{\Omega_0}$ onto the discrete space of cell $\Omega_0$, for $i=1,2$.
  \item[E*.] Compute $\vvU^{(\vkappa,i)}_{\Omega_0}$: the projection of $\vvu^{(\vkappa,i)}_{\Omega_0}$ onto $\mathrm{span}\{\vvU_{h,\Omega_0}^{(\vkappa,1)}, \vvU_{h,\Omega_0}^{(\vkappa,2)}\}$, for $i=1,2$.
  \item[F*.] Compute $B_{ij}:=(\vvu_{\Omega_0}^{(\vkappa,i)},\vvu_{\Omega_0}^{(\vkappa,j)})_{M_0}$ and $R_{ij}:=(\vvu^{(\vkappa,i)}_{\Omega_0}-\vvU_{\Omega_0}^{(\vkappa,i)},\vvu^{(\vkappa,j)}_{\Omega_0}-\vvU_{\Omega_0}^{(\vkappa,j)})_{M_0}$, for $i,j=1,2$, and use this to compute $e_{vec}(\vkappa)=\sqrt{\sigma_{max}(B^{-1}R)}$.
\end{enumerate}

So far, we only considered the dispersion error and eigenvector error for a given wave vector $\vkappa$. For a given wavelength $\lambda=2\pi/|\vkappa|$, we define the dispersion and eigenvector error as the worst case among all wave vectors of length $|\vkappa|=\lambda/(2\pi)$, so among wave vectors in all possible directions:
\begin{subequations}
\begin{align}
e_{disp}(\lambda) &:= \sup_{\vkappa\in\mathbb{R}^3, \;|\vkappa|=\lambda/(2\pi)} e_{disp}(\vkappa), \\
 e_{vec}(\lambda) &:= \sup_{\vkappa\in\mathbb{R}^3, \;|\vkappa|=\lambda/(2\pi)} e_{vec}(\vkappa).
\end{align}
\label{eq:errLambda}%
\end{subequations}
We can use these errors to determine the required number of elements per wavelength. To compute these errors, we use a search algorithm, which requires the computation of $e_{disp}(\vkappa)$ and $e_{vec}(\vkappa)$ for a large number of wave vectors $\vkappa$. The complete procedure for computing the dispersion and eigenvector error is given by:
\begin{enumerate}
  \item Construct a cell $\Omega_0$ and subdivide it into tetrahedra.
  \item Compute the submatrices $M^{(\Omega_0)}$ and $A^{(\Omega_0,\Omega_{\Delta \vct{k}})}$ for $\Delta\vct{k}\in\{-1,0,1\}^3$.
  \item Compute $s_{h,max}$, given by (\ref{eq:spectralBound}). This is done with a search algorithm which requires the computation of $ \sigma_{max}(S^{(\vkappa)})$, with $S^{(\vkappa)}:=M^{(\Omega_0)}_{inv}A^{(\vkappa)}$, for a large number of wave vectors $\vkappa$.
  \item Compute $\Delta t\leq \sqrt{c_K/s_{h,max}}$, with $c_K$ given by (\ref{eq:cK}).
  \item For a given wavelength $\lambda$, compute the errors $e_{disp}(\lambda)$ and $e_{vec}(\lambda)$ given in (\ref{eq:errLambda}). For each $\lambda$, this requires the computation of $e_{disp}(\vkappa)$ and $e_{vec}(\vkappa)$, using steps a-d and A-F, for a large number of wave vectors $\vkappa$.
\end{enumerate}

\subsection{Estimating the Computational Cost}
\label{sec:numComp}
To compare the efficiency of the different methods, we also compute the number of degrees of freedom $n_{vec}$, the number of non-zero entries of the stiffness matrix $n_{mat}$, and the estimated computational cost $n_{comp}$, for each wavelength $\lambda$.

We define $n_{vec}$ to be the number of degrees of freedom per $\lambda^3$-volume. This is computed by
\begin{align*}
n_{vec} &= n_0\frac{\lambda^3}{|\Omega_0|},
\end{align*}
where $n_0$ is the number of basis functions corresponding to cell $\Omega_0$,  and $|\Omega_0|$ is the volume of $\Omega_0$. 

We define $n_{mat}$ to be the number of non-zero entries of the stiffness matrix per $\lambda^3$-volume. In case of mass-lumping, we estimate this number by
\begin{align*}
n^{(ML)}_{mat} &= \left(\sum_{q\in\mathcal{Q}_{\Omega_0}} \sum_{q'\in\mathcal{N}(q)} |\mathcal{U}_q||\mathcal{U}_{q'}| \right) \frac{\lambda^3}{|\Omega_0|} ,
\end{align*}
where $|\mathcal{U}_q|$ is the number of degrees of freedom per node ($|\mathcal{U}_q|=1$ in the acoustic and $|\mathcal{U}_q|=3$ in the elastic case), $\mathcal{Q}_{\Omega_0}$ is the set of nodes on $\Omega_0$, and $\mathcal{N}(q)$ are the neighbouring nodes of $q$ that are connected with $q$ through an element. 

In case of the SIPDG method, we estimate this number by
\begin{align*}
n^{(DG)}_{mat} &= \left(\sum_{e\in\mathcal{T}_{\Omega_0}}\sum_{e'\in\mathcal{N}(e)}|\mathcal{U}_e||\mathcal{U}_{e'}| \right) \frac{\lambda^3}{|\Omega_0|} ,
\end{align*}
where $|\mathcal{U}_e|$ is the number of basis functions with support on element $e$, $\mathcal{T}_{\Omega_0}$ are the elements in $\Omega_0$, and $\mathcal{N}(e)$ are the neighbouring elements of $e$ that are connected with $e$ through a face.

To estimate the computational cost we look at the size of the matrix times the number of matrix-vector products. The resulting estimates gives a rough estimate of the relative CPU time of the different methods, since it estimates the number of computations when using a globally assembled matrix. 

We define the computational cost $n_{comp}$ as the number of non-zero matrix entries per $\lambda^3$-volume times the number of matrix-vector products during one oscillation in time. The duration of one oscillation is $T_0=\lambda/c$, with $c$ the wave propagation speed. The number of matrix-vector products during one oscillation is the number of stages of the Lax--Wendroff scheme $K$ times the number of time steps $N_{\Delta t}=T_0/\Delta t=\lambda/(c\Delta t)$, where $\Delta t=\sqrt{c_K/s_{h,max}}$, with $c_K$ given by (\ref{eq:cK}) and $s_{h,max}$ given by (\ref{eq:spectralBound}). We use this to compute $n_{comp}$ as follows:
\begin{align*}
n_{comp} &= n_{mat} K N_{\Delta t}.
\end{align*}

\section{Results and Comparisons}
\label{sec:results}

\begin{table}[h]
\caption{Analysed finite element methods }
\label{tab:methods}
\begin{center}
\begin{tabular}{l || p{0.75\textwidth}}
Method & Description \\ \hline\hline
ML1& Linear mass-lumped  finite element method \\ \hline
ML2 & Degree-$2$ mass-lumped finite element method \cite{mulder96} \\ \hline
ML3a, ML3b & Degree-$3$ mass-lumped finite element methods \cite{chin99} \\ \hline
DGX & Symmetric Interior Penalty Discontinous Galerkin method \cite{grote06} of degree $X=1,2,3$ \\ \hline
DGXa & DGX with penalty term derived in \cite{geevers17} and given by (\ref{eq:penTerm_a}) \\ \hline
DGXb & DGX with penalty term derived in \cite{mulder14} and given by (\ref{eq:penTerm_b}) 
\end{tabular}
\end{center}
\end{table}

An overview of the different finite element methods that we analyse is given in Table \ref{tab:methods}. Each method is combined with an order-$2p$ Lax--Wendroff time integration scheme, where $p$ denotes the degree of the spatial discretization.

\begin{figure}[h]
\centering
\begin{subfigure}[b]{0.45\textwidth}
  \includegraphics[width=\textwidth]{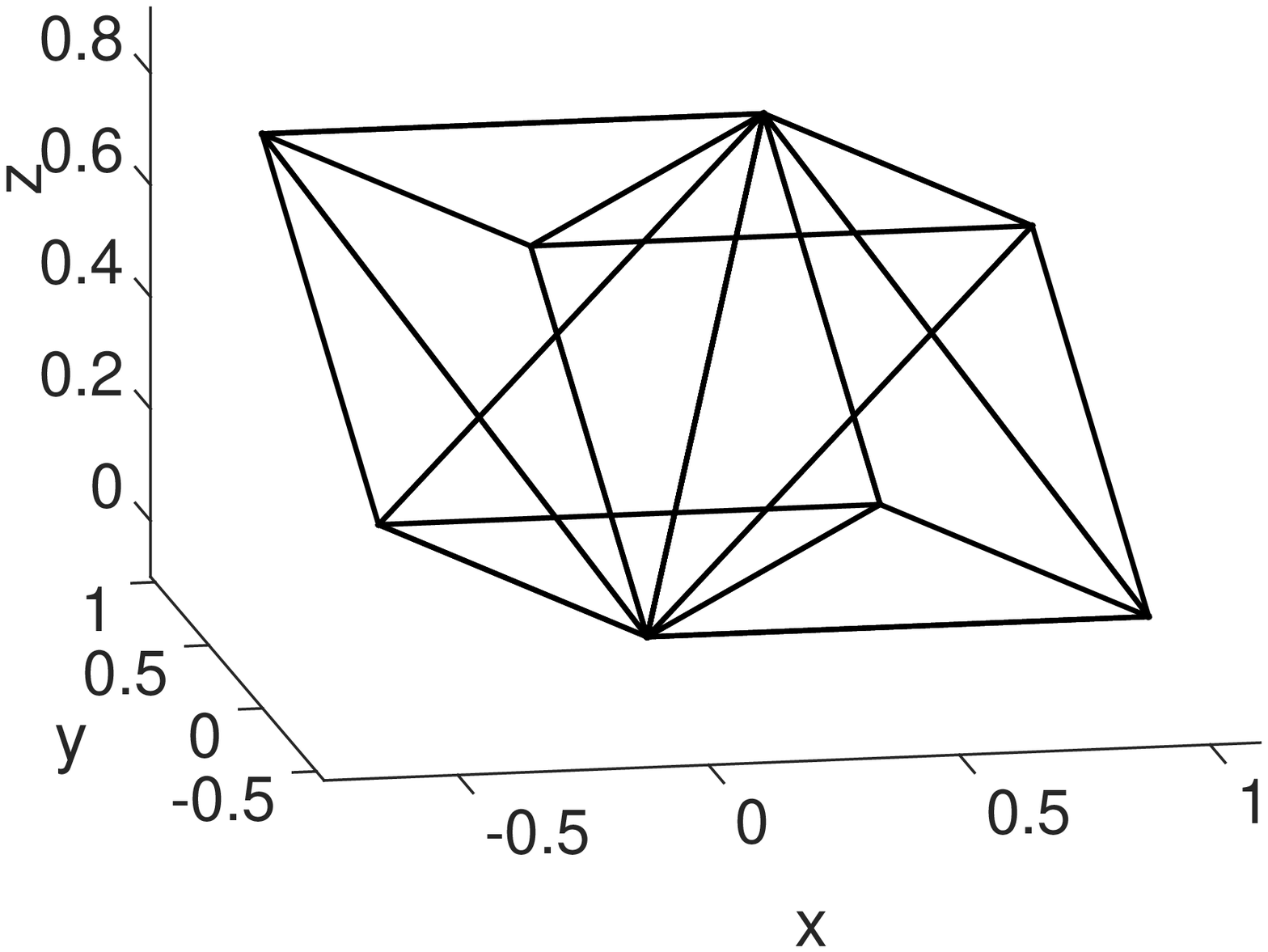}
\end{subfigure} \,\,
\begin{subfigure}[b]{0.45\textwidth}
  \includegraphics[width=\textwidth]{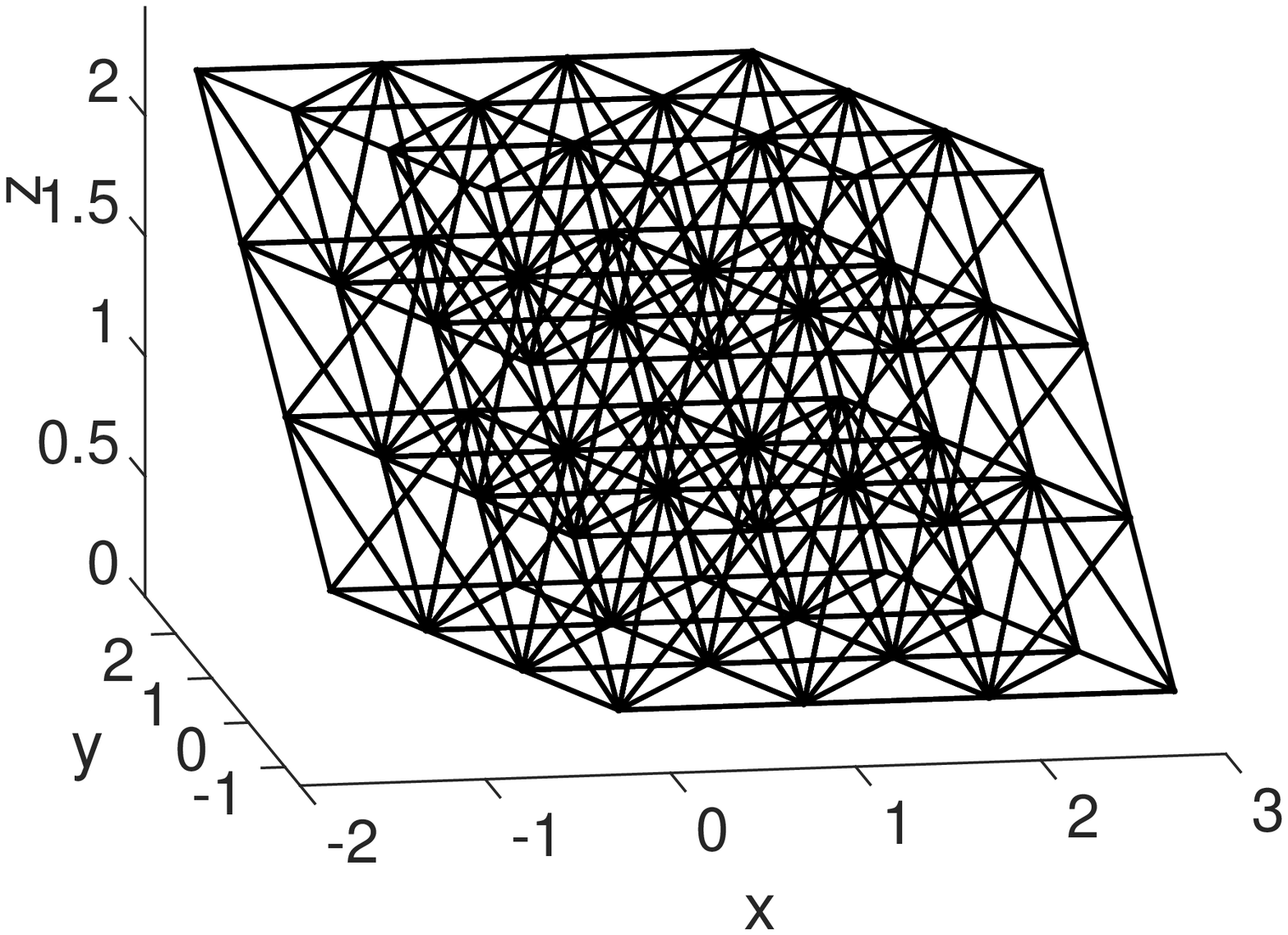}
\end{subfigure} 
\caption{Tetragonal disphenoid honeycomb restricted to cell $\Omega_0$ (left) and restricted to $3\times 3\times 3$ cells (right).}
\label{fig:tetra0}
\end{figure}

To analyse the dispersion properties of these methods, we use standard Fourier analysis, as explained in Section \ref{sec:dispAn}. We consider a periodic mesh of congruent nearly-regular equifacial tetrahedra, known as the tetragonal disphenoid honeycomb. To obtain this mesh, we slice the unit cell $\Omega_0:=[0,1)^3$ into 6 tetrahedra with the planes $x=y$, $x=z$, and $y=z$ and then apply the linear transformation $\vx\rightarrow\ten{T}\cdot\vx$, with 
\begin{align}
\ten{T} := \begin{bmatrix}
1 & -1/3 & -1/3 \\
0 & \sqrt{8/9} & -\sqrt{2/9} \\
0 & 0 & \sqrt{2/3}
\end{bmatrix}.
\label{eq:defT}
\end{align}
An illustration of this mesh is given in Figure \ref{fig:tetra0}.

\subsection{Acoustic Waves on a Regular Mesh}
We first consider the acoustic wave model with $c=\rho=1$. Figure \ref{fig:errAc1} illustrates the dispersion and eigenvector error with respect to the number of elements per wavelength $N_E:=\sqrt[3]{\lambda^3/|e|_{av}}$, with $\lambda$ the wavelength and $|e|_{av}$ the average element volume. The eigenvector error for ML1 is always zero, since it has only one degree of freedom per cell $\Omega_{\vct{k}}$ and therefore allows only one numerical plane wave for a given wave vector. From this figure we can obtain the order of convergence, which is $2p$ for the dispersion error and $p+1$ for the eigenvector error. These convergence rates are typical for symmetric finite element methods for eigenvalue problems, see, for example, \cite{boffi10} and the references therein. The $2p$-order superconvergence of the dispersion error is also in accordance with the results of \cite{mulder99, ainsworth06, deBasabe08}.

\begin{figure}[h]
\centering
\begin{subfigure}[b]{0.45\textwidth}
 \includegraphics[width=\textwidth]{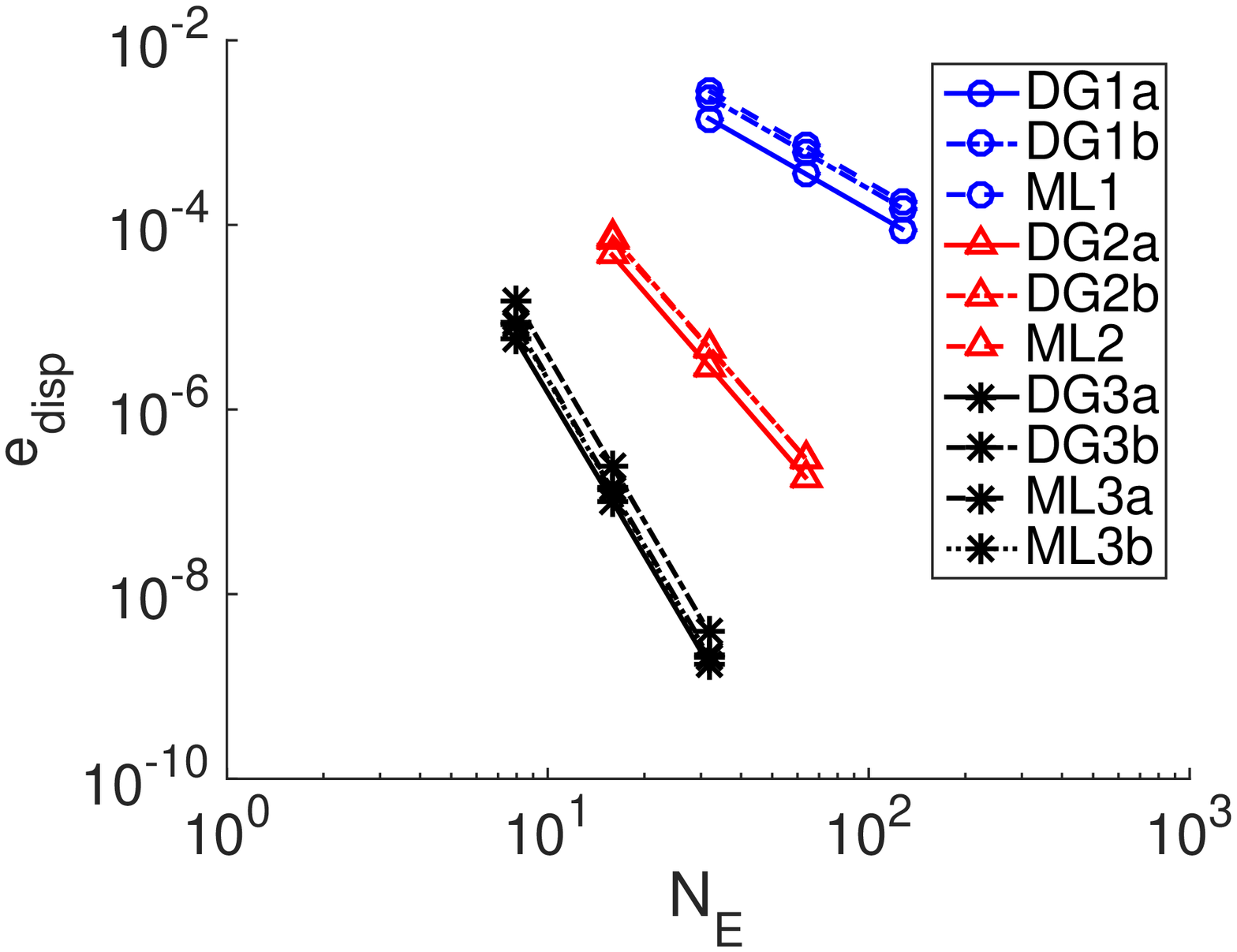}
 \end{subfigure} \,\,
\begin{subfigure}[b]{0.45\textwidth}
  \includegraphics[width=\textwidth]{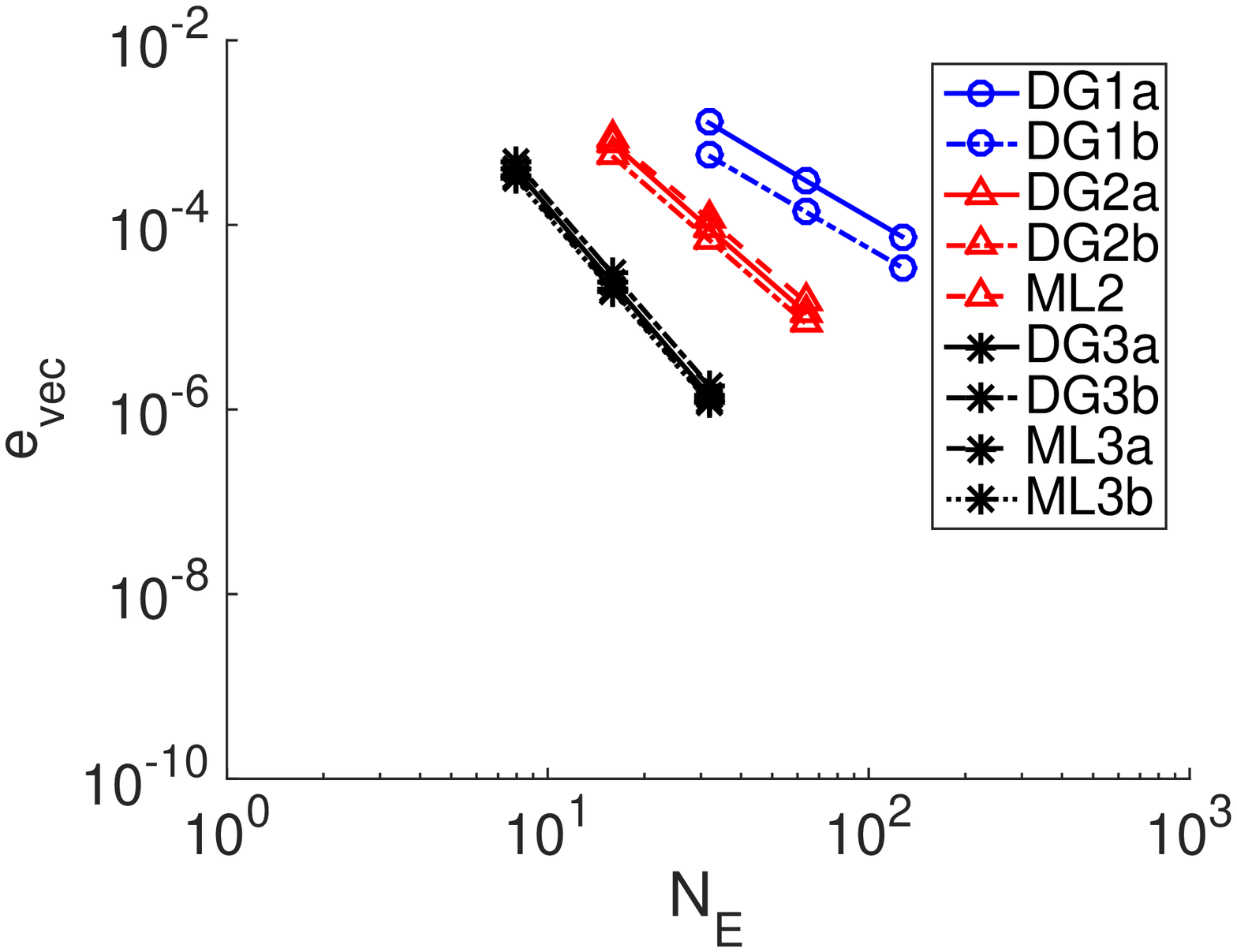}
\end{subfigure}
\caption{Dispersion error (left) and eigenvector error (right) for the acoustic wave model.}
\label{fig:errAc1}
\end{figure}

By extrapolating the results shown in Figure \ref{fig:errAc1} we can obtain approximations of the errors of the form $e=\alpha (N_E)^{-\beta}$, where $\alpha$ is the leading constant and $\beta$ is the order of convergence. The approximations are given in Table \ref{tab:errAc1}. 

\begin{table}[h]
\caption{Approximation of the dispersion and eigenvector error for the acoustic case.}
\label{tab:errAc1}
\begin{center}
{\tabulinesep=0.5mm
\begin{tabu}{c||c|c|}
Method	& $e_{disp}$ 	& $e_{vec}$  \\ \hline\hline
DG1a 	& $1.45(N_E)^{-2}$ & $1.20(N_E)^{-2}$ \\
DG1b 	& $2.46(N_E)^{-2}$ & $0.56(N_E)^{-2}$ \\
ML1 		& $2.87(N_E)^{-2}$ & $0$      \\ \hline
DG2a  	& $3.00(N_E)^{-4}$ & $2.89(N_E)^{-3}$ \\
DG2b 	& $4.83(N_E)^{-4}$ & $2.22(N_E)^{-3}$ \\
ML2  	& $4.82(N_E)^{-4}$ & $3.78(N_E)^{-3}$ \\ \hline
DG3a  	& $1.77(N_E)^{-6}$ & $1.46(N_E)^{-4}$ \\
DG3b  	& $3.98(N_E)^{-6}$ & $1.88(N_E)^{-4}$ \\
ML3a  	& $2.25(N_E)^{-6}$ & $1.26(N_E)^{-4}$ \\
ML3b  	& $2.15(N_E)^{-6}$ & $1.22(N_E)^{-4}$ \\
\end{tabu}}
\end{center}
\end{table}

We can use these results to obtain estimates for the number of elements per wavelength required for a given accuracy, but we can also use them to obtain other properties, such as the number of time steps or the computational cost required for a given accuracy. An overview for a dispersion error of $0.01$ and $0.001$ is given in Table \ref{tab:errDispAc1} and \ref{tab:errDispAc01}, respectively, and the relation between the accuracy and the computational cost is illustrated in Figure \ref{fig:errDispAcComp1}. 

\begin{figure}[h]
\centering
\includegraphics[width=0.6\textwidth]{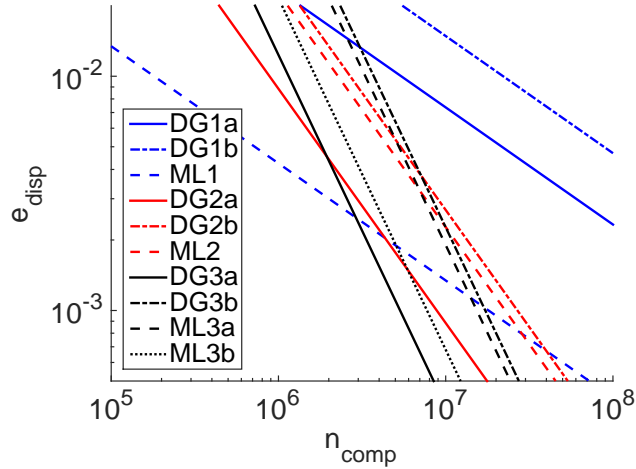}
\caption{Dispersion error of different finite element methods for the acoustic wave model plotted against the estimated computational cost.}
\label{fig:errDispAcComp1}
\end{figure}

\begin{table}[h]
\caption{Number of elements per wavelength $N_E$, number of degrees of freedom $n_{vec}$, size of the global matrix $n_{mat}$, number of time steps $N_{\Delta t}$, computational cost $n_{comp}$ and eigenvector error $e_{vec}$ for a dispersion error of $0.01$ for different finite element methods for the acoustic wave model. The numbers are accurate up to two decimal places. }
\label{tab:errDispAc1}
\begin{center}
{\tabulinesep=0.5mm
\begin{tabu}{l|| r|r|r|r|r|| l}
  & \multicolumn{6}{c}{$e_{disp}=0.01$}  \\ \hline
Method	& $N_E$ 	& $n_{vec}$ & $n_{mat}$ 		& $N_{\Delta t}$ & $n_{comp}$  	& $e_{vec}$ \\ \hline\hline
DG1a 	& $12\phantom{.0}$  & $7000$	& $140\times 10^3$ 	& $39$ 	& $5.4\phantom{0}\times 10^{6}$ 	& $0.0083$  \\
DG1b 	& $16\phantom{.0}$  &$15000$	& $310\times 10^3$ 	& $72$ 	& $22\phantom{.00}\times 10^6$ 	& $0.0023$  \\
ML1 		& $17\phantom{.0}$  & $810$	& $12\times 10^3$ 	& $15$ 	& $0.18\times 10^6$  			& $0$ \\ \hline
DG2a  	& $4.2$ 			& $720$  	& $36\times 10^3$ 	& $12$ 	& $0.88\times 10^6$ 			& $0.040$ \\
DG2b 	& $4.7$ 			& $1000$ 	& $51\times 10^3$ 	& $26$ 	& $2.7\phantom{0}\times 10^6$ 	& $0.022$ \\
ML2  	& $4.7$ 			& $860$ 	& $39\times 10^3$ 	& $29$ 	& $2.3\phantom{0}\times 10^6$ 	& $0.037$  \\ \hline
DG3a  	& $2.4$ 			& $270$ 	& $27\times 10^3$ 	& $14$ 	& $1.1\phantom{0}\times 10^6$ 	& $0.046$ \\
DG3b  	& $2.7$ 			& $400$ 	& $40\times 10^3$ 	& $31$ 	& $3.7\phantom{0}\times 10^6$ 	& $0.035$ \\
ML3a  	& $2.5$ 			& $370$ 	& $31\times 10^3$ 	& $36$ 	& $3.3\phantom{0}\times 10^6$ 	& $0.034$ \\
ML3b  	& $2.4$ 			& $360$ 	& $30\times 10^3$ 	& $18$ 	& $1.7\phantom{0}\times 10^6$ 	& $0.034$ \\
\end{tabu}}
\end{center}
\end{table}

\begin{table}[h]
\caption{Same as Table \ref{tab:errDispAc1}, but for a dispersion error of $0.001$. }
\label{tab:errDispAc01}
\begin{center}
{\tabulinesep=0.5mm
\begin{tabu}{l|| r|r|r|r|r|| l}
  & \multicolumn{6}{c}{$e_{disp}=0.001$}  \\ \hline
Method	& $N_E$ 	& $n_{vec}$ & $n_{mat}$ 		& $N_{\Delta t}$ & $n_{comp}$  	& $e_{vec}$ \\ \hline\hline
DG1a 	& $38\phantom{.0}$  & $220000$	& $4400\times 10^3$& $120$ 	& $540\phantom{.0}\times 10^6$ 	& $0.00083$  \\
DG1b 	& $50\phantom{.0}$  & $490000$	& $9700\times 10^3$& $230$ 	& $2200\phantom{.0}\times 10^6$ 	& $0.00023$  \\
ML1 		& $54\phantom{.0}$  & $26000$	& $390\times 10^3$ 	& $47$ 	& $18\phantom{.0}\times 10^6$  	& $0$ \\ \hline
DG2a  	& $7.4$ 			& $4100$  	& $200\times 10^3$ 	& $22$ 	& $8.8\times 10^6$ 				& $0.0071$ \\
DG2b 	& $8.3$ 			& $5800$ 		& $290\times 10^3$ 	& $46$ 	& $27\phantom{.0}\times 10^6$ 	& $0.0038$ \\
ML2  	& $8.3$ 			& $4800$ 		& $220\times 10^3$ 	& $52$ 	& $23\phantom{.0}\times 10^6$ 	& $0.0065$  \\ \hline
DG3a  	& $3.5$ 			& $840$ 		& $84\times 10^3$ 	& $21$ 	& $5.3\times 10^6$ 				& $0.010$ \\
DG3b  	& $4.0$ 			& $1300$ 		& $130\times 10^3$ 	& $46$ 	& $17\phantom{.0}\times 10^6$ 	& $0.0075$ \\
ML3a  	& $3.6$ 			& $1200$ 		& $98\times 10^3$ 	& $52$ 	& $15\phantom{.0}\times 10^6$ 	& $0.0074$ \\
ML3b  	& $3.6$ 			& $1100$ 		& $96\times 10^3$ 	& $27$ 	& $7.7\times 10^6$ 				& $0.0073$ \\
\end{tabu}}
\end{center}
\end{table}

Figure \ref{fig:errDispAcComp1} shows that for linear elements, the mass-lumped method ML1 is significantly more efficient than the DG methods DG1a and DG1b, while for quadratic elements, DG2a is significantly more efficient than ML2 and DG2b, and for cubic functions, DG3a is slightly more efficient than ML3b and significantly more efficient than DG3b and ML3a. In all cases, the DG methods using the sharper penalty term given by (\ref{eq:penTerm_a}) are significantly more efficient than those using the penalty term given by (\ref{eq:penTerm_b}). For a dispersion error of around $0.01$ and higher, the linear mass-lumped method ML1 performs best in terms of computational cost, while for a dispersion error below $0.001$ the best method is the DG method with cubic basis functions DG3a or the second degree-$3$ mass-lumped finite element method ML3b.

Tables \ref{tab:errDispAc1} and \ref{tab:errDispAc01} also show that for the case $p=1$, the eigenvector error is always smaller than the dispersion error, but that for higher-order elements, the eigenvector error can become almost $5$ times as large when the dispersion error is $0.01$ and $10$ times as large when the dispersion error is $0.001$. This is due to the fact that the dispersion error converges with a faster rate (order $2p$) than the eigenvector error (order $p+1$) for higher-degree methods.

\subsection{The Effect of Mesh Distortions}
We also investigate the effect of the mesh quality on the dispersion error. To do this, we first create meshes of very flat elements by scaling the regular disphenoid mesh in the $z$-direction. After that, we create distorted meshes by displacing some of the vertices of the disphenoid honeycomb.

\begin{figure}[h]
\centering
\begin{subfigure}[b]{0.45\textwidth}
  \includegraphics[width=\textwidth]{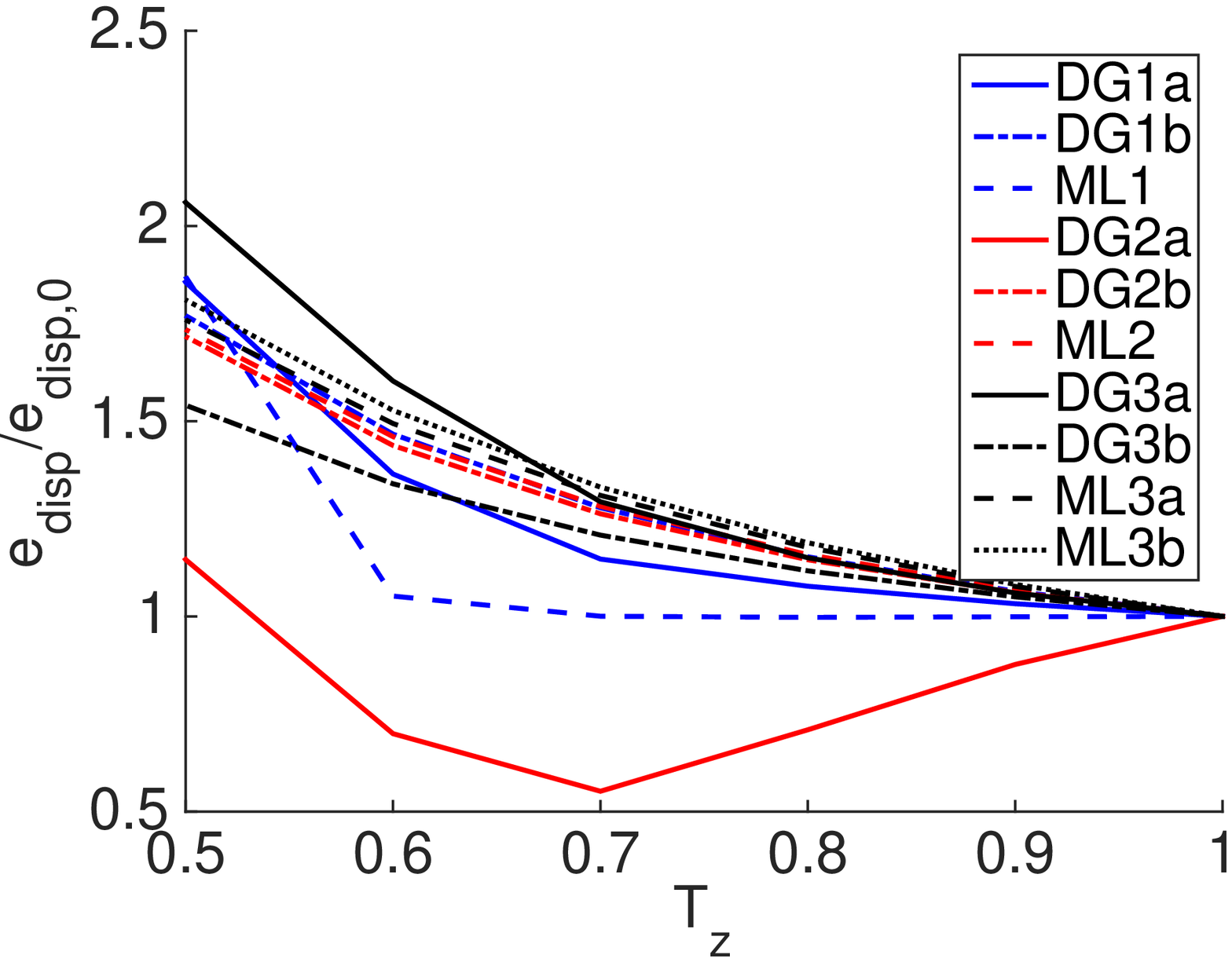}
\end{subfigure} \,\,
\begin{subfigure}[b]{0.45\textwidth}
  \includegraphics[width=\textwidth]{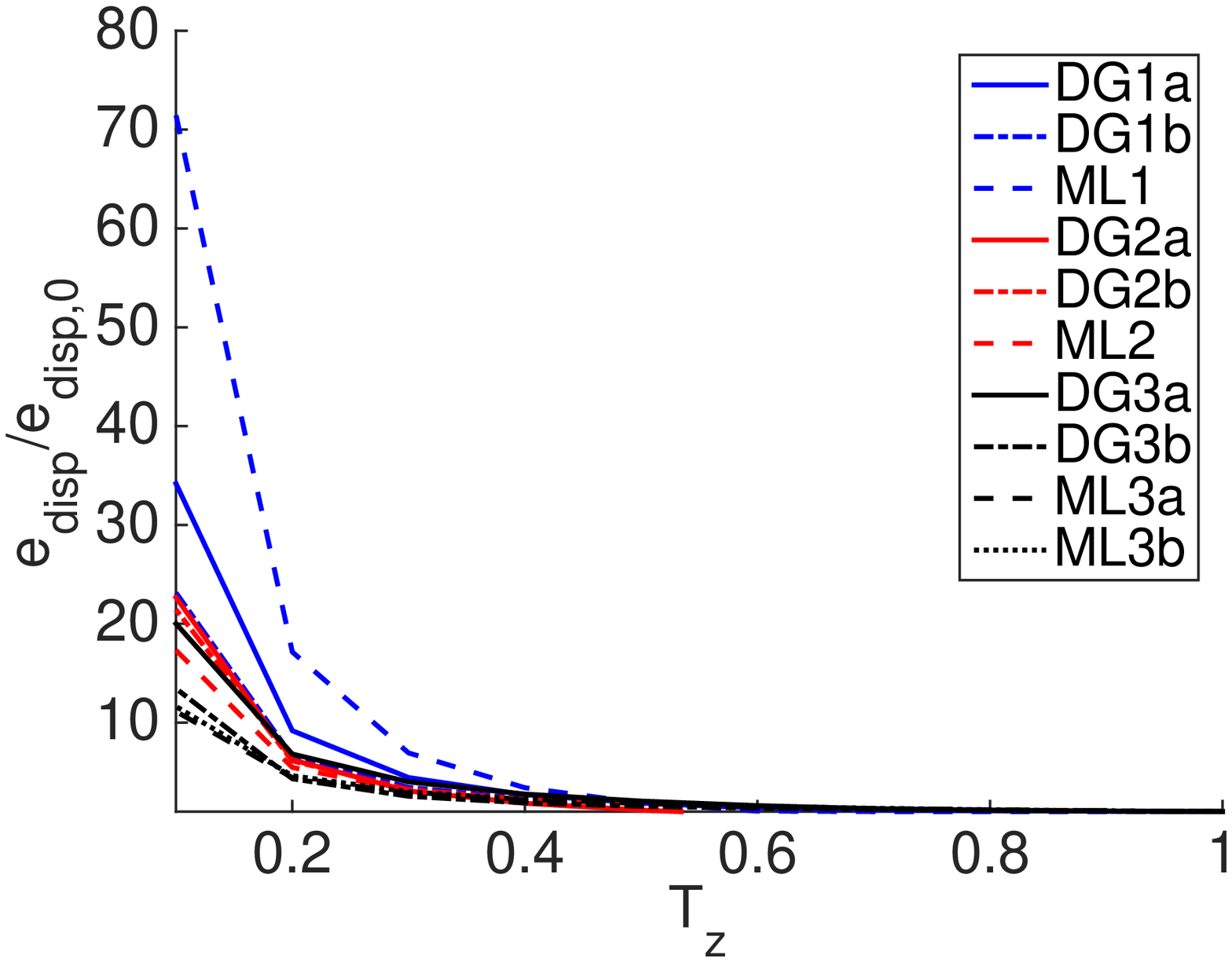}
\end{subfigure}
\caption{Relative dispersion error for the disphenoid mesh scaled in the $z$-direction by a factor $T_z$ for the acoustic wave model. Here, $e_{disp,0}$ denotes the error for the original mesh.}
\label{fig:errRelAc}
\end{figure}

To create flat elements, we scale the disphenoid mesh in the $z$-direction by a factor $T_z$. The effect on the dispersion error is illustrated in Figure \ref{fig:errRelAc}. For a mesh flattened by a factor $2$, the dispersion error does not grow more than a factor $2.5$, but flattening the mesh by a factor $10$ increases the error by a factor between $10$ and $100$. In all cases, the mesh resolution remains the same and even becomes smaller in the $z$-direction. This means that the mesh quality can have a strong effect on the accuracy of the method and that using flat tetrahedra can significantly reduce the accuracy. The methods using lower-order elements are more sensitive to the mesh quality than the higher-order methods.

\begin{figure}[h]
\centering
\begin{subfigure}[b]{0.45\textwidth}
  \includegraphics[width=\textwidth]{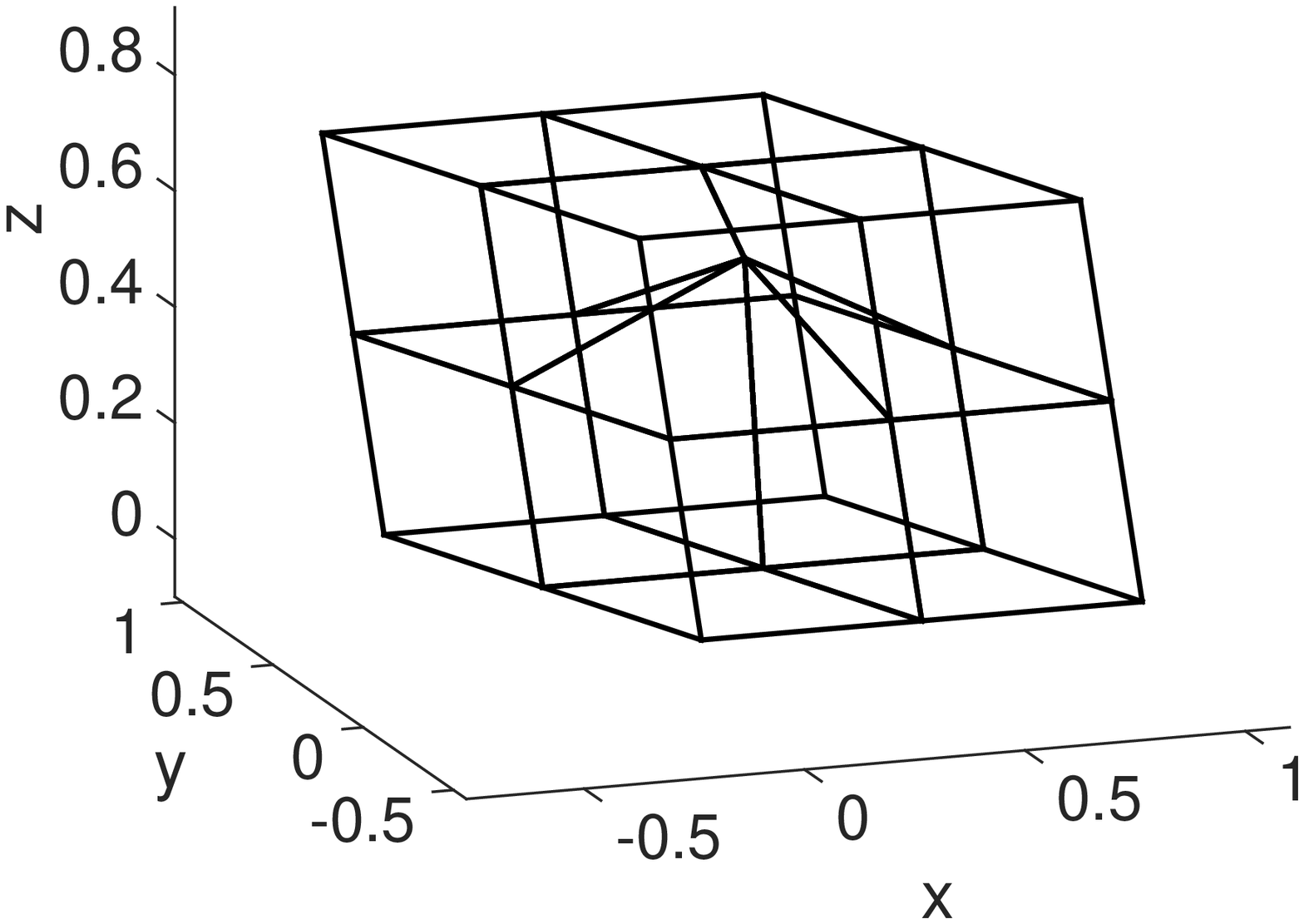}
\end{subfigure} \,\,
\begin{subfigure}[b]{0.45\textwidth}
  \includegraphics[width=\textwidth]{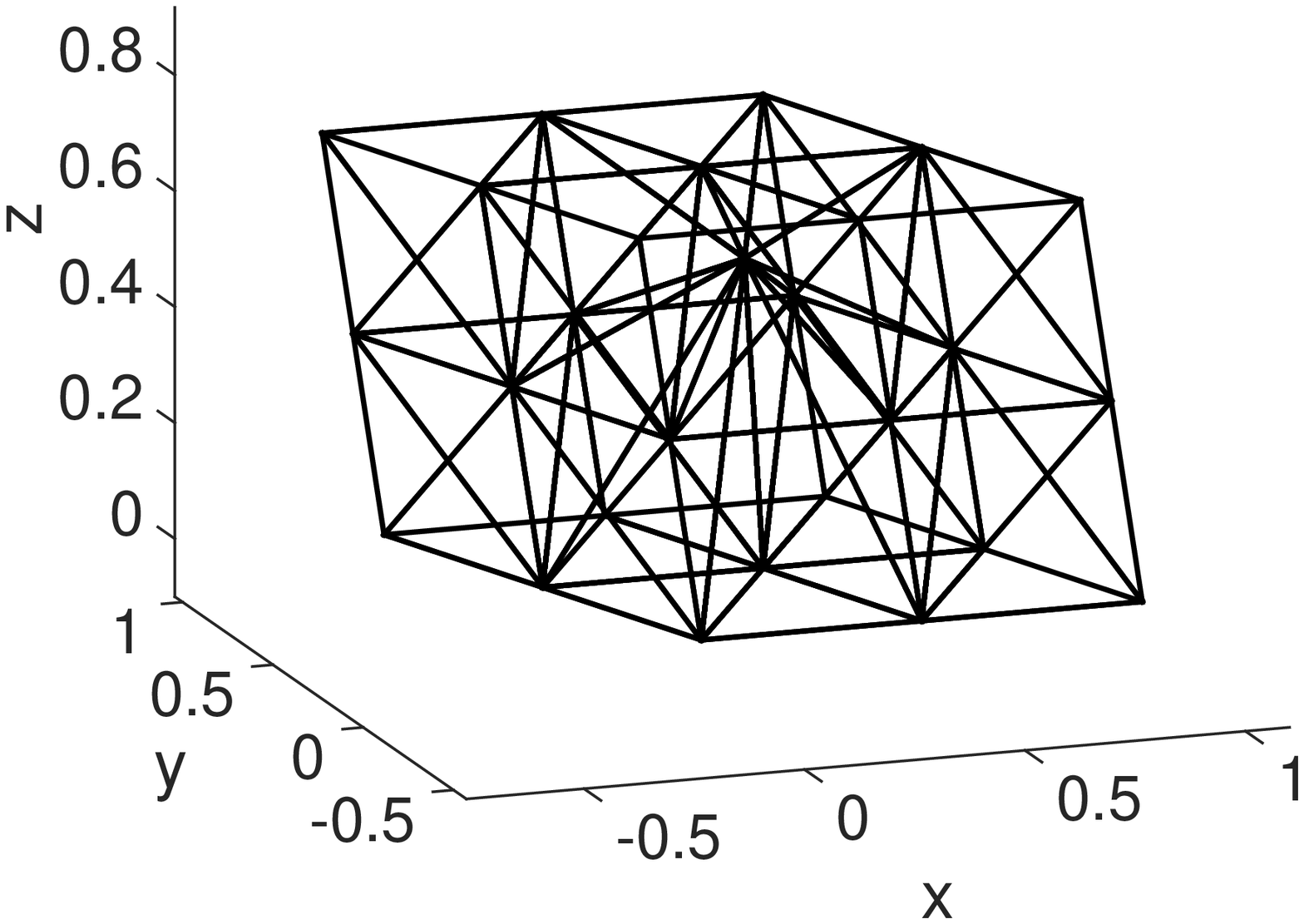}
\end{subfigure}
\caption{Repeated subcells with a small distortion (left) and corresponding tetrahedral mesh (right).}
\label{fig:meshDist}
\end{figure}

\begin{figure}[h]
\centering
\begin{subfigure}[b]{0.45\textwidth}
  \includegraphics[width=\textwidth]{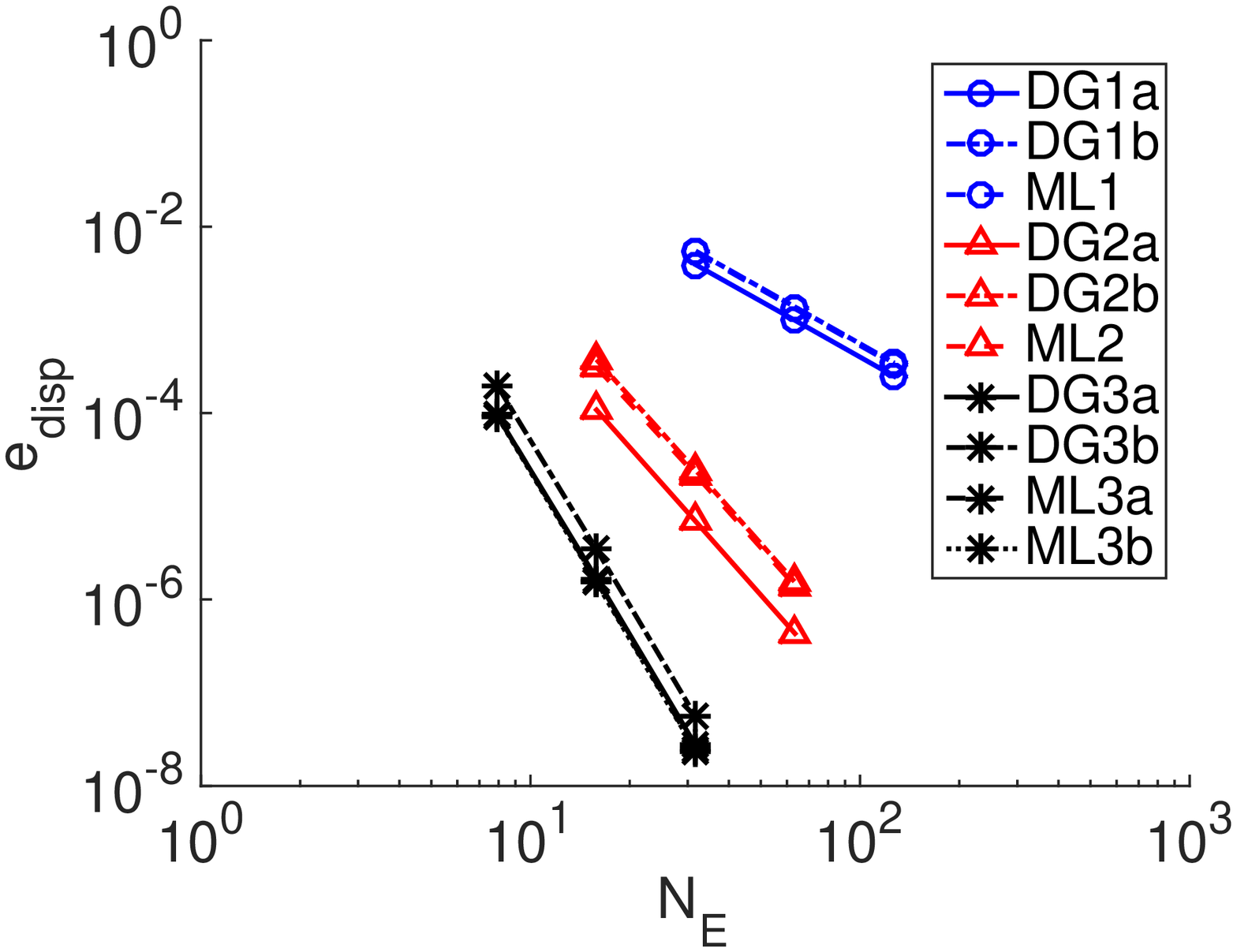}
\end{subfigure} \,\,
\begin{subfigure}[b]{0.45\textwidth}
  \includegraphics[width=\textwidth]{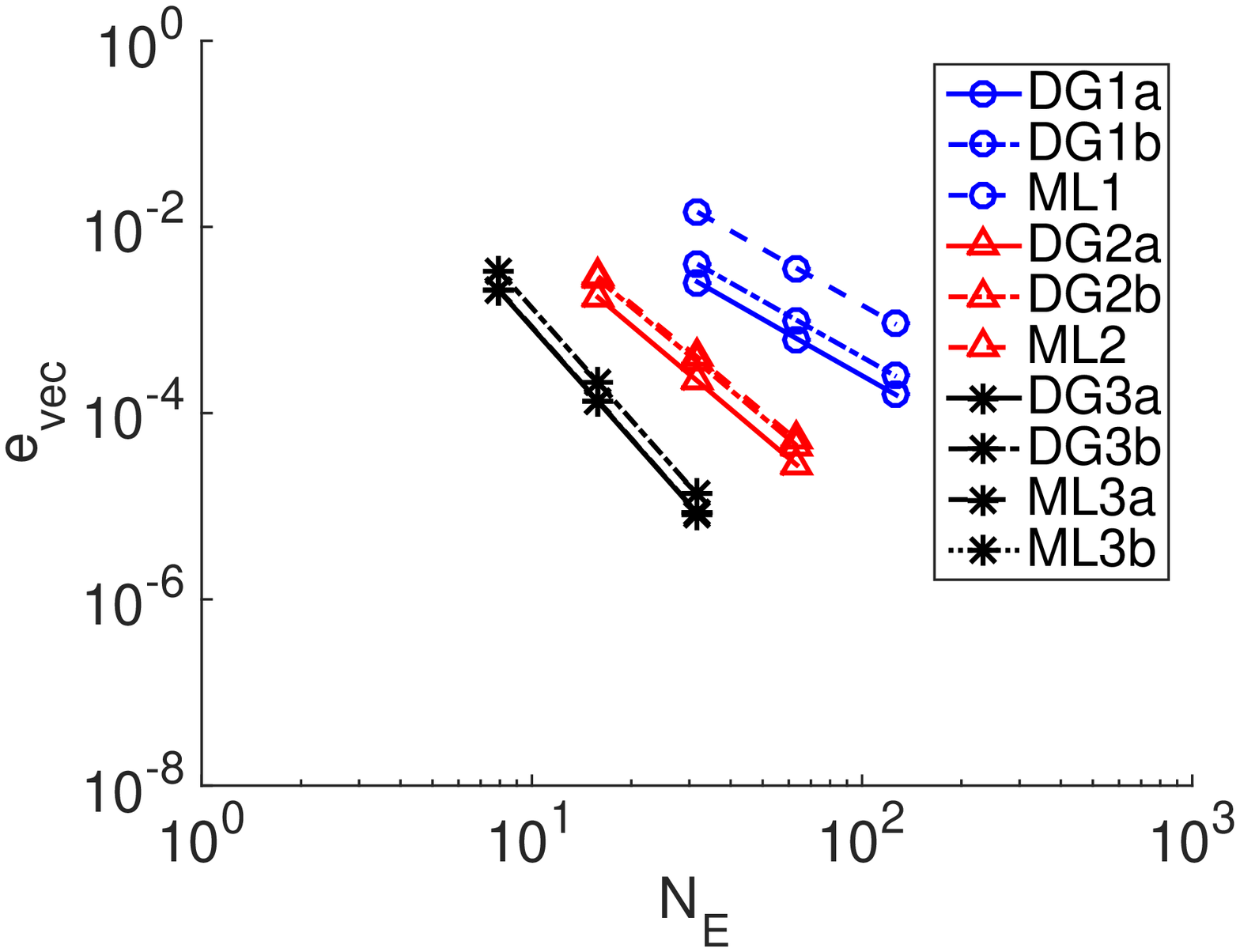}
\end{subfigure} 
\caption{Dispersion error (left) and eigenvector error (right) for the acoustic wave model for a distorted mesh with distortion $\delta=0.9$.}
\label{fig:errAcDist04}
\end{figure}

\begin{figure}[h]
\centering
\begin{subfigure}[b]{0.45\textwidth}
  \includegraphics[width=\textwidth]{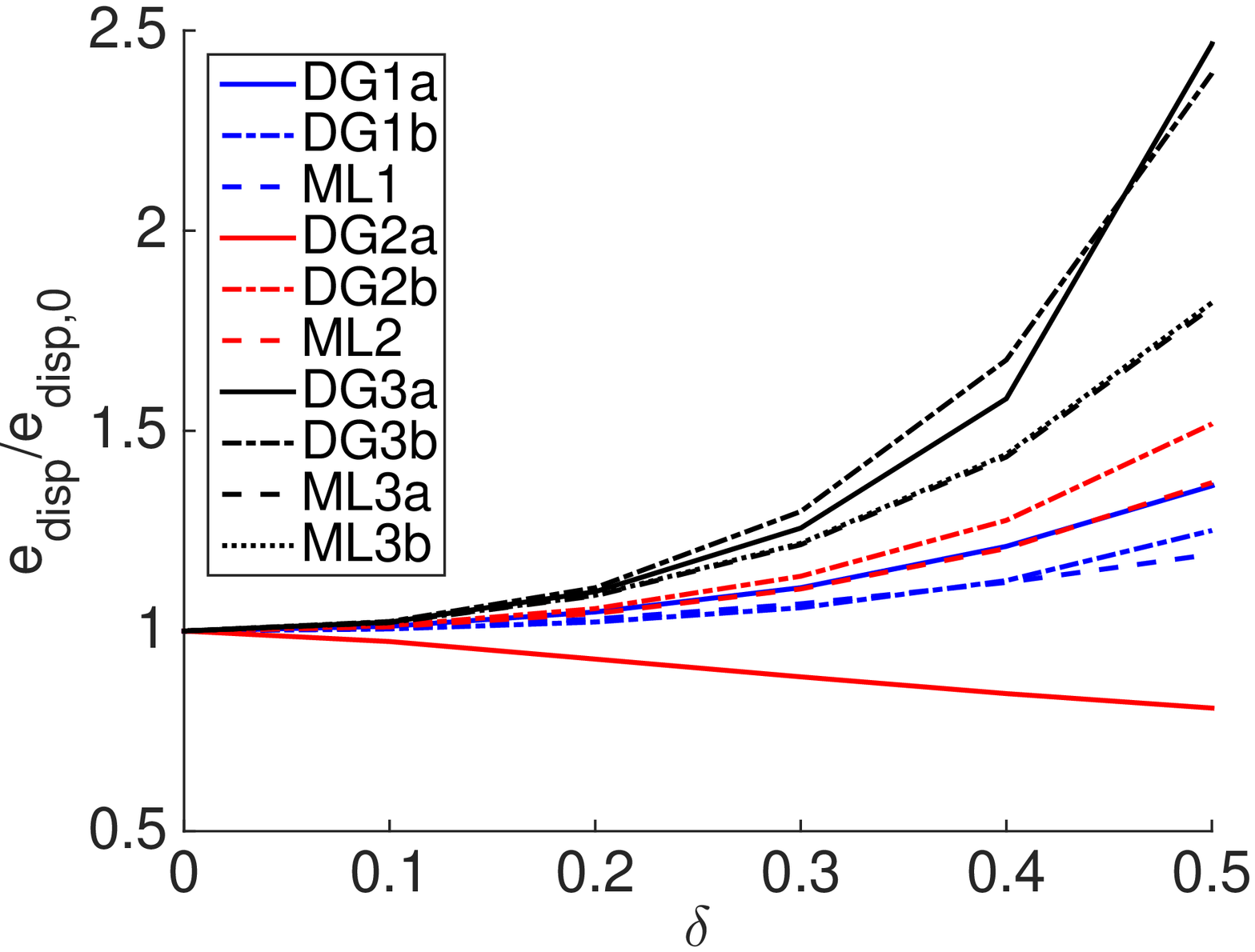}
\end{subfigure} \,\,
\begin{subfigure}[b]{0.45\textwidth}
  \includegraphics[width=\textwidth]{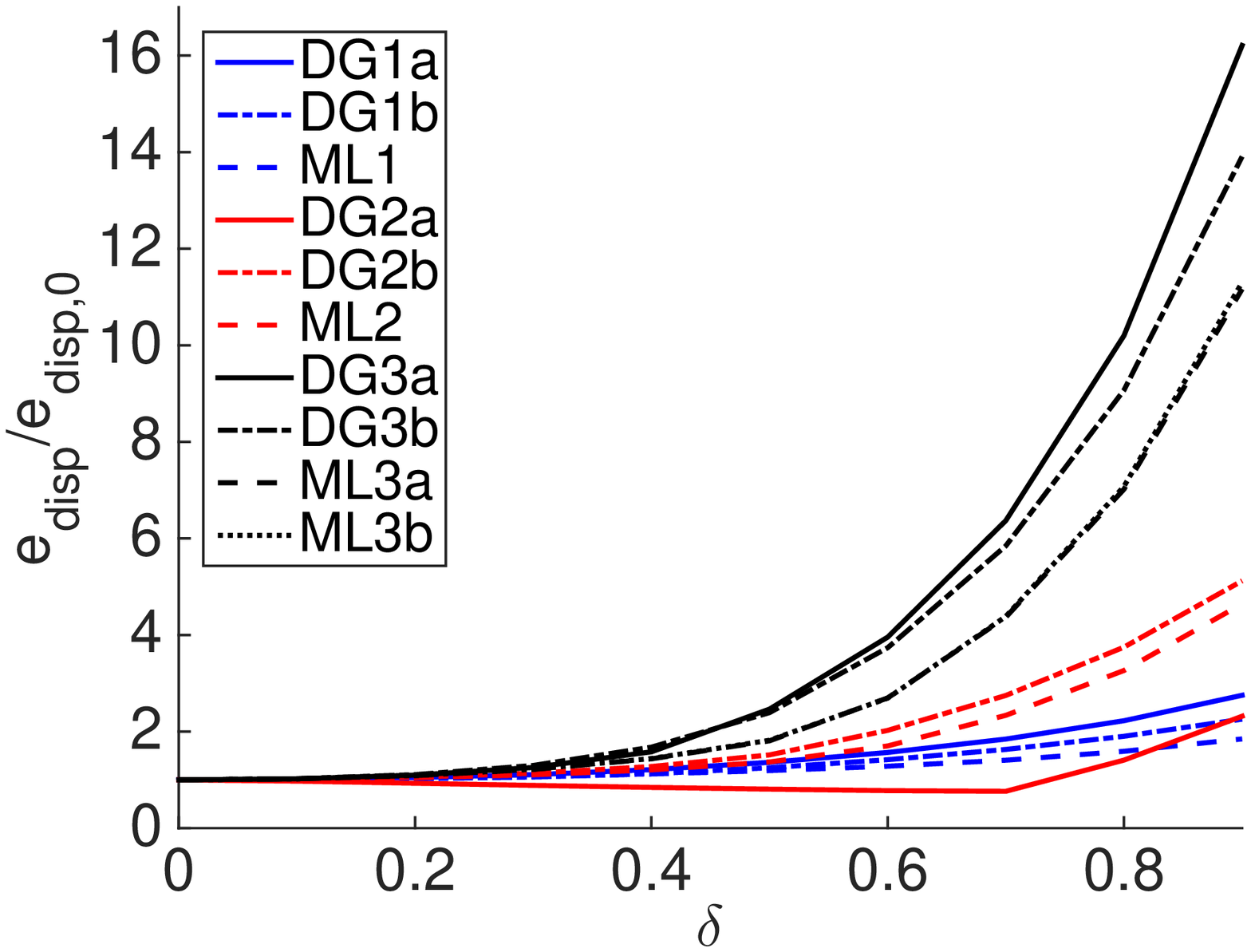}
\end{subfigure}
\caption{Relative dispersion error for meshes with a distortion $\delta$. Here, $e_{disp,0}$ denotes the error of the regular mesh with $\delta = 0$.}
\label{fig:errRelAcDist}
\end{figure}

To create distorted meshes, we displace some of the vertices of the disphenoid mesh. In particular, we create a distorted mesh using the following steps:
\begin{enumerate}
  \item Slice the cube $[0,0.5)^3$ into 6 tetrahedra with the planes $x=y$, $x=z$, $y=z$.
  \item Repeat this pattern $2\times 2\times 2$ times to pack the unit cell $[0,1)$ with 48 tetrahedra.
  \item Displace the central node by moving it from $(0.5,0.5,0.5)$ to $\big(0.5(1+\delta),0.5(1+\delta),0.5(1+\delta)\big)$, where $\delta\in[0,1)$ denotes the size of the distortion.
  \item Apply the transformation $\vx\rightarrow\ten{T}\cdot\vx$, with $\ten{T}$ defined as in (\ref{eq:defT}).
\end{enumerate}
In case of zero distortion, $\delta=0$, we obtain the original disphenoid honeycomb, scaled by a factor $0.5$. When the distortion $\delta$ approaches $1$, some of the elements become completely flat with zero volume.

An illustration of the mesh with distortion $\delta=0.4$ is given in Figure \ref{fig:meshDist}. In Figure \ref{fig:errAcDist04}, the dispersion and eigenvector error are plotted against the number of elements per wavelength for a heavily distorted mesh with $\delta=0.9$. These results show that the order of convergence remains $2p$ for the dispersion and $p+1$ for the eigenvector error, even though the mesh is distorted. The distortion does, however, affect the leading constant of the errors. The effect of the mesh distortion on the dispersion error is illustrated in Figure \ref{fig:errRelAcDist}. Again, the accuracy is not significantly affected by small distortions, but large distortions can reduce the accuracy by an order of magnitude.

\subsection{Elastic Waves and the Effect of the P/S-wave Velocity Ratio}

Besides the acoustic wave model, we also consider the isotropic elastic wave model. Figure \ref{fig:errIso1} illustrates the dispersion and eigenvector error with respect to the number of elements per wavelength for the isotropic elastic wave model with $\mu=\rho=1$ and $\lambda=2$, so with a P/S-wave velocity ratio of $2$. Again, the order of convergence is $2p$ for the dispersion error and $p+1$ for the eigenvector error. 

By extrapolating these results we can again obtain approximations of the errors of the form $e=\alpha (N_E)^{-\beta}$, which are given in Table \ref{tab:errIso1}. Figure \ref{fig:errDisptIsoComp1} illustrates the relation between the dispersion error and the computational cost, based on these results. The relative performance of the different methods is similar to the acoustic case.

\begin{figure}[h]
\centering
\begin{subfigure}[b]{0.45\textwidth}
  \includegraphics[width=\textwidth]{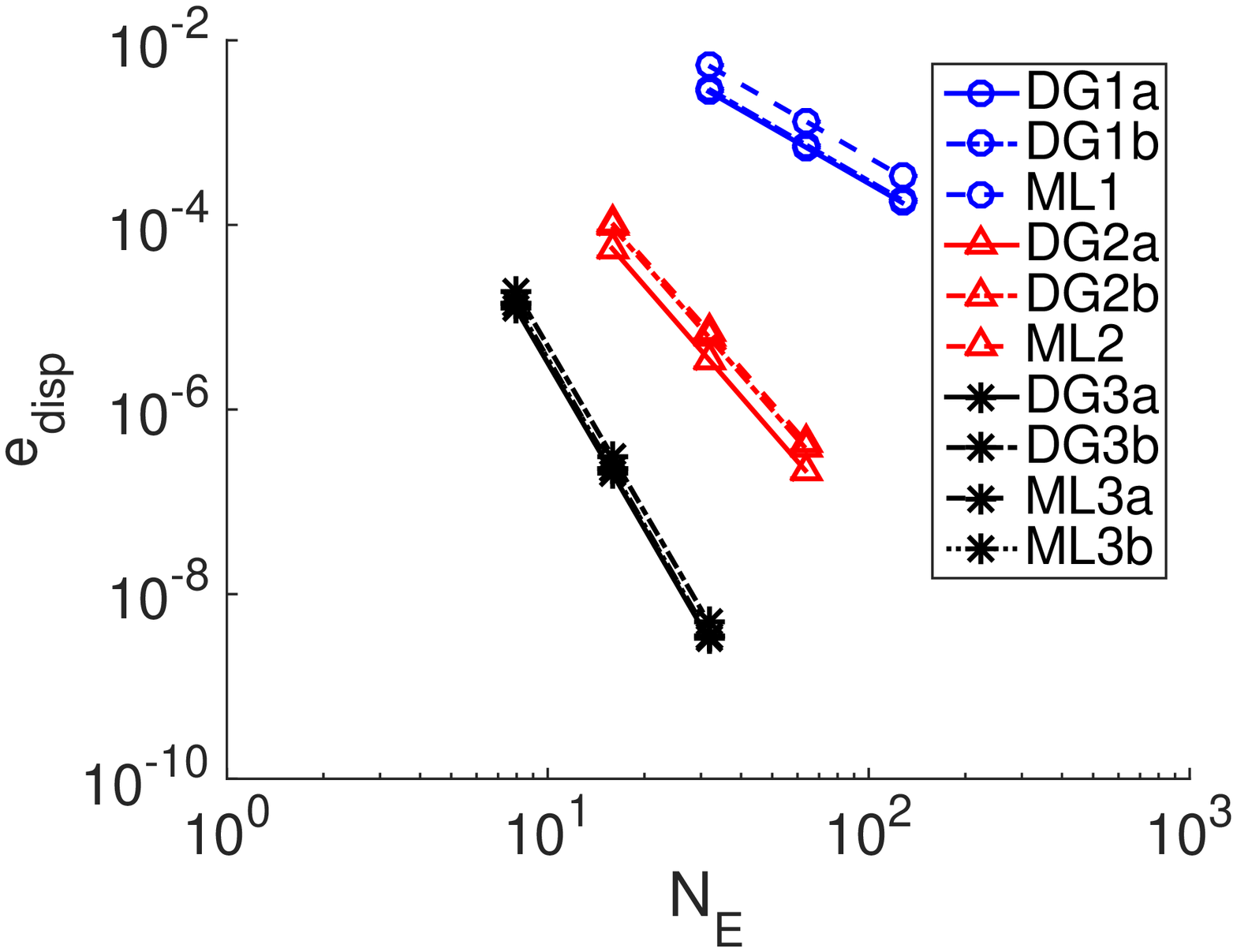}
\end{subfigure} \,\,
\begin{subfigure}[b]{0.45\textwidth}
  \includegraphics[width=\textwidth]{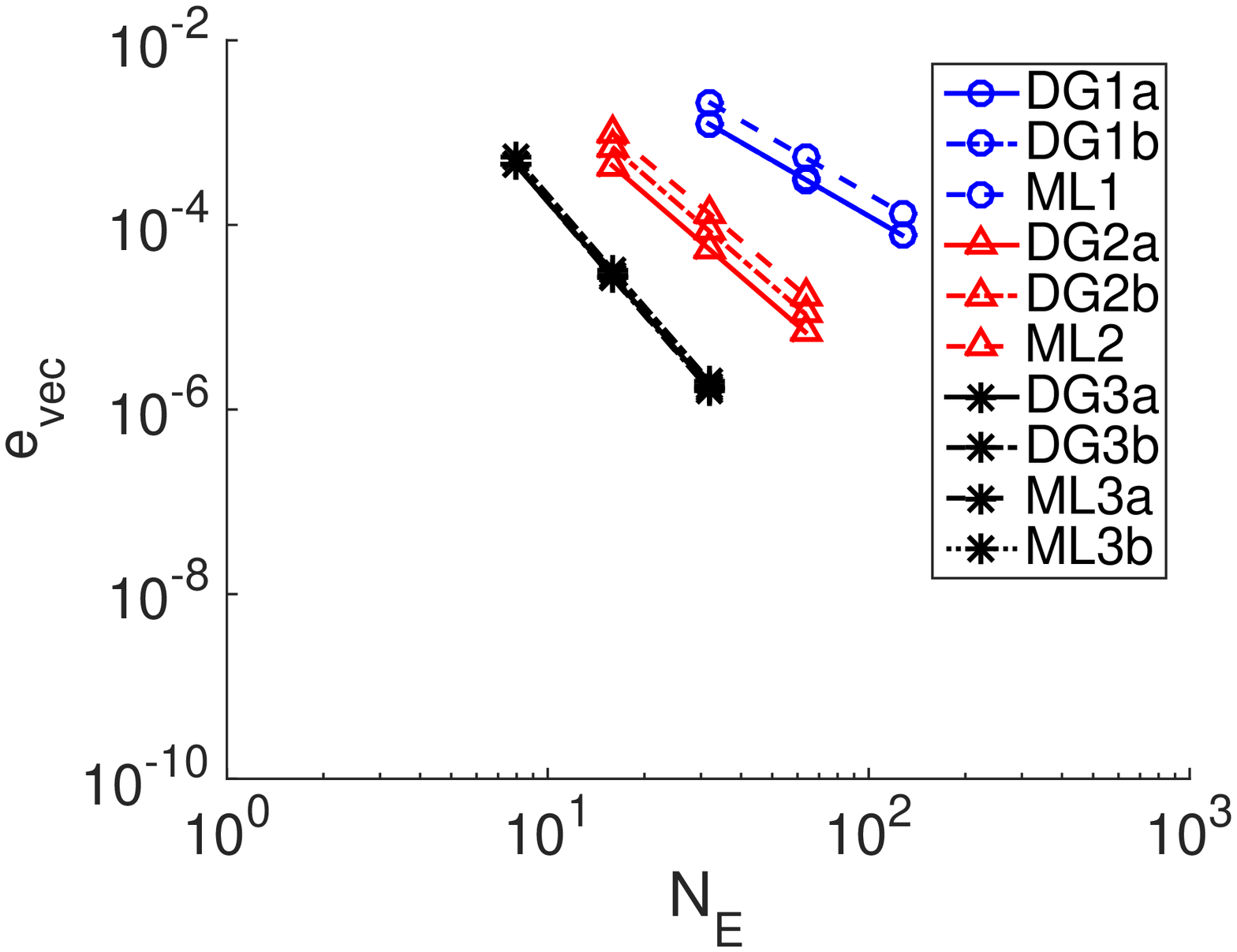}
\end{subfigure}
\caption{Dispersion error (left) and eigenvector error (right) for the isotropic elastic wave model with a P/S-wave velocity ratio of $2$.}
\label{fig:errIso1}
\end{figure}

\begin{table}[h]
\caption{Approximation of the dispersion and eigenvector error for the elastic wave model with a P/S-wave velocity ratio of $2$.}
\label{tab:errIso1}
\begin{center}
{\tabulinesep=0.5mm
\begin{tabu}{c||c|c|}
Method	& $e_{disp}$ 	& $e_{vec}$  \\ \hline\hline
DG1a 	& $2.81(N_E)^{-2}$ & $1.25(N_E)^{-2}$ \\
DG1b 	& $3.00(N_E)^{-2}$ & $1.25(N_E)^{-2}$ \\
ML1 		& $5.39(N_E)^{-2}$ & $2.16(N_E)^{-2}$      \\ \hline
DG2a  	& $3.55(N_E)^{-4}$ & $1.76(N_E)^{-3}$ \\
DG2b 	& $6.20(N_E)^{-4}$ & $2.77(N_E)^{-3}$ \\
ML2  	& $7.29(N_E)^{-4}$ & $4.39(N_E)^{-3}$ \\ \hline
DG3a  	& $3.32(N_E)^{-6}$ & $1.79(N_E)^{-4}$ \\
DG3b  	& $5.04(N_E)^{-6}$ & $2.11(N_E)^{-4}$ \\
ML3a  	& $3.63(N_E)^{-6}$ & $1.66(N_E)^{-4}$ \\
ML3b  	& $3.58(N_E)^{-6}$ & $1.69(N_E)^{-4}$ \\
\end{tabu}}
\end{center}
\end{table}

\begin{figure}[h]
\centering
\includegraphics[width=0.6\textwidth]{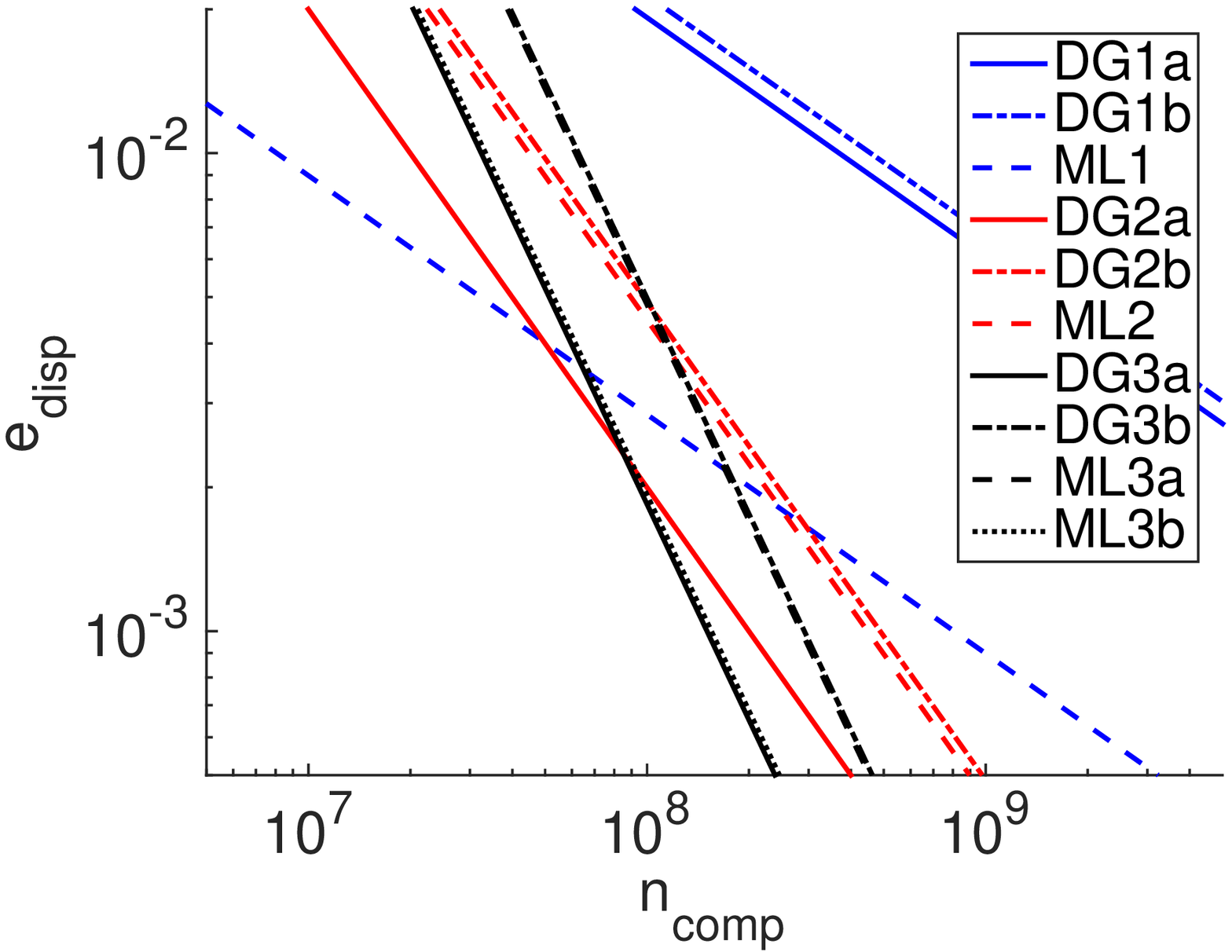}
\caption{Dispersion error of different finite element methods for the isotropic elastic wave model with a P/S-wave velocity ratio of $2$, plotted against the estimated computational cost. The graphs of DG3a and ML3b and of DG3b and ML3a are almost identical.}
\label{fig:errDisptIsoComp1}
\end{figure}

We also look at the influence of the P/S-wave velocity ratio $c_P/c_S$ on the dispersion error, where $c_S=\sqrt{\mu}$ denotes the S-wave velocity and $c_P=\sqrt{\lambda+2\mu}$ denotes the P-wave velocity. This relation is illustrated in Figure \ref{fig:errRelIsoC10}. This figure shows that the DG methods are not really sensitive to the $c_P/c_S$ ratio, since the dispersion error never grows more than a factor $1.5$. The higher-order mass-lumped methods are slightly more sensitive, with a dispersion error becoming around $3$ times as large for $c_P/c_S=10$, compared to $c_P/c_S=2$, while the linear mass-lumped method is very sensitive, with a dispersion error becoming almost $40$ times as large in this case.

\begin{figure}[h]
\centering
\begin{subfigure}[b]{0.45\textwidth}
  \includegraphics[width=\textwidth]{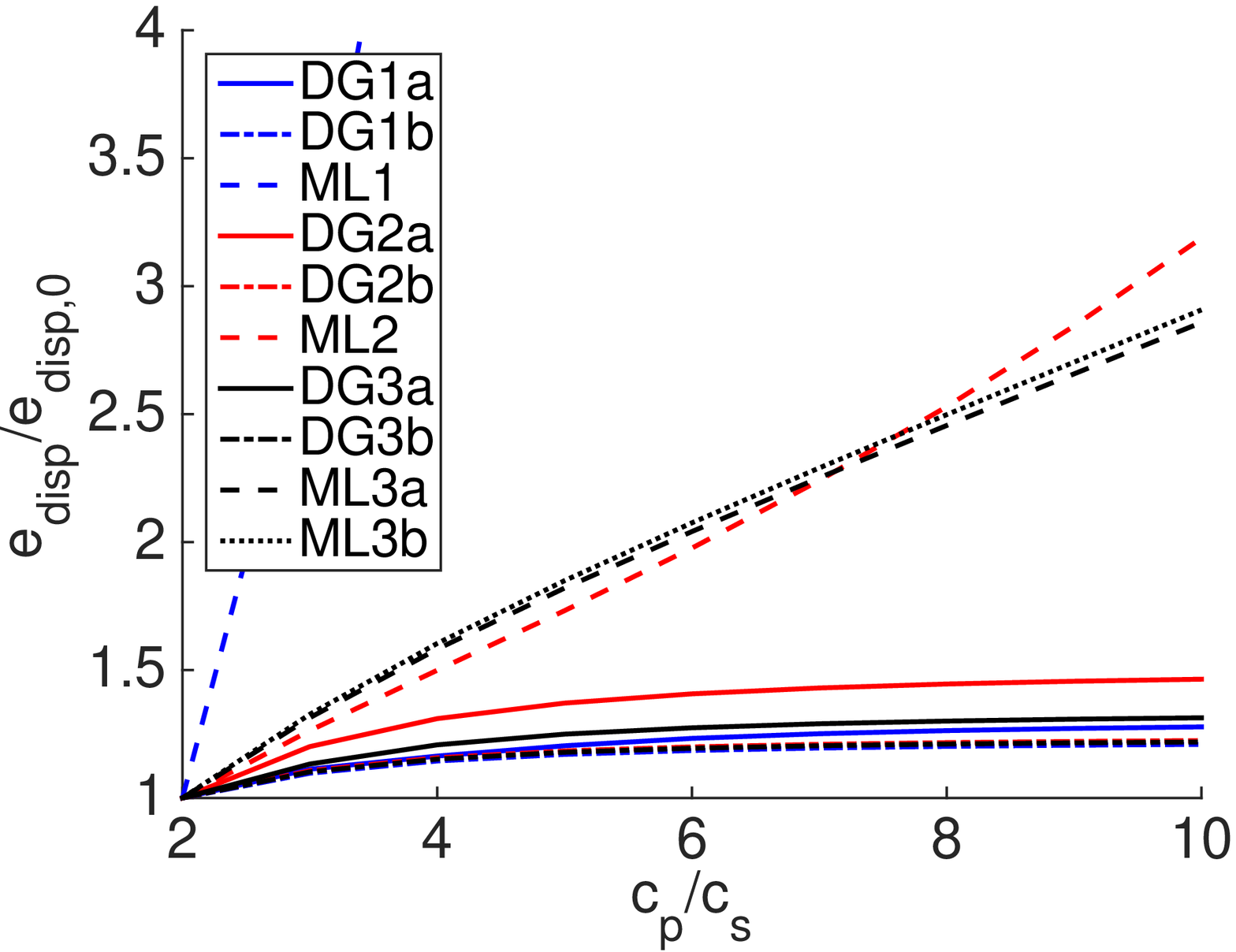}
\end{subfigure} \,\,
\begin{subfigure}[b]{0.45\textwidth}
  \includegraphics[width=\textwidth]{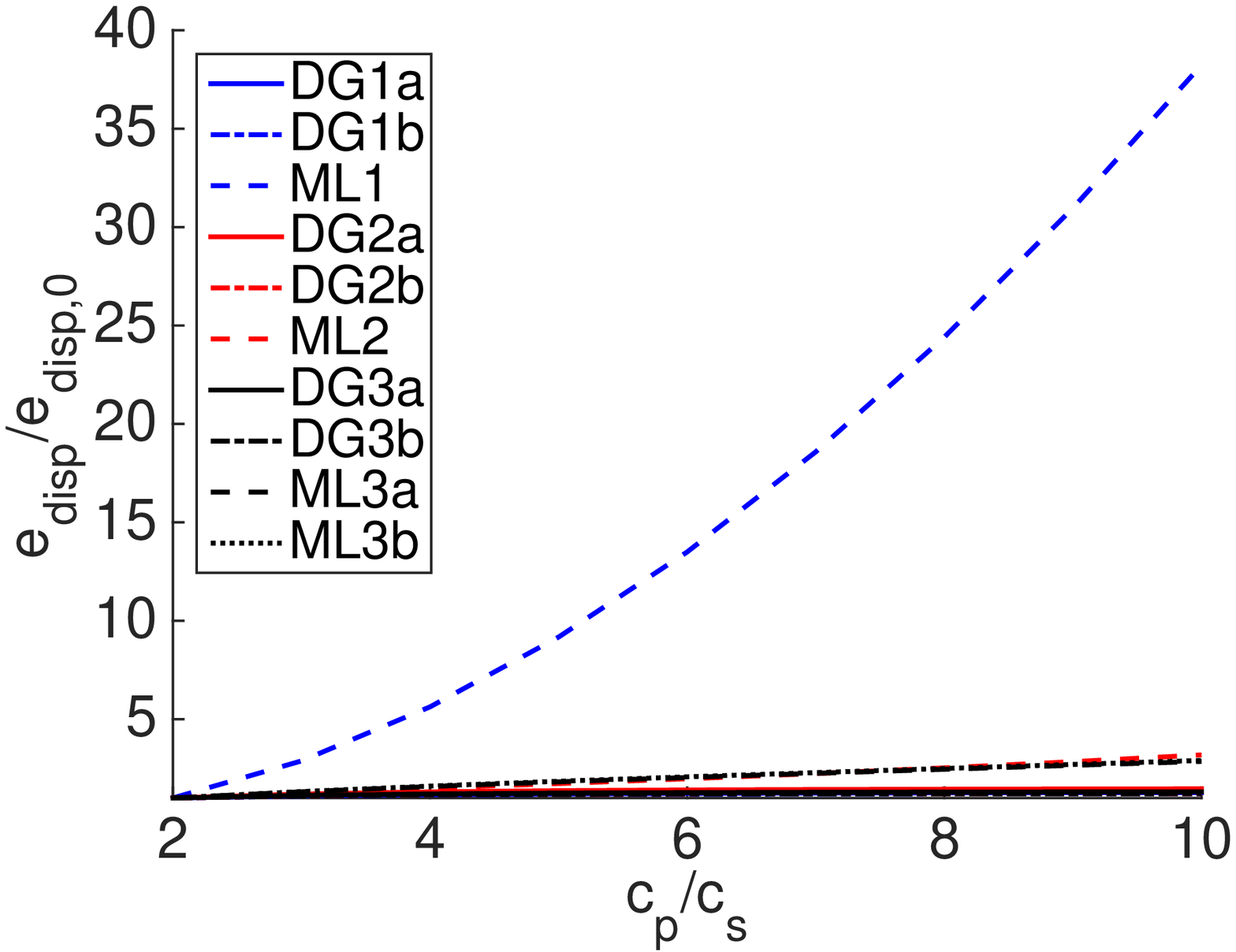}
\end{subfigure} 
\caption{Relative dispersion error for the isotropic elastic wave model with different $c_P/c_S$ ratios. Here, $e_{disp,0}$ denotes the error for the original mesh with $c_P/c_S=2$.}
\label{fig:errRelIsoC10}
\end{figure}

\section{Conclusions}
\label{sec:conclusions}
We analysed the dispersion properties of two types of explicit finite element methods for modelling wave propagation on tetrahedral meshes, namely mass-lumped finite elements methods and symmetric interior penalty discontinuous Galerkin (SIPDG) methods, both for degrees $p=1,2,3$ and combined with an order-$2p$ Lax--Wendroff time integration method. The analysed methods are listed in Table 1.

The dispersion properties are obtained semi-analytically using standard Fourier analysis. We used this to give an indication of which method is the most efficient for a given accuracy, how many elements per wavelength are required for a given accuracy, and how sensitive the accuracy of the method is to poorly shaped elements and high P/S-wave velocity ratios. 

Based on the results we draw the following conclusions with regard to efficiency:
\begin{itemize}
  \item The linear mass-lumped method is the most efficient method for a dispersion error of around $1\%$ when using approximately regular tetrahedra. Heavily distorted elements, however, can significantly reduce its accuracy.
  \item The degree-$3$ SIPDG method, with the penalty term derived in \cite{geevers17} and given by (\ref{eq:penTerm_a}), and the second degree-$3$ mass-lumped finite element method of \cite{chin99} are the most efficient methods for a dispersion error of around $0.1\%$ and less.
  \item The SIPDG methods using the sharper penalty term bound derived in \cite{geevers17} are significantly more efficient than those using the penalty term of \cite{mulder14}, which is based on the trace inequality of \cite{warburton03}. 
\end{itemize}

The required number of elements for a given accuracy can be obtained from the approximations given in Tables \ref{tab:errAc1} and \ref{tab:errIso1}. We also draw the following conclusions with regard to accuracy:
\begin{itemize}
  \item Higher-order methods suffer more from spurious modes for the same dispersion error. This is due to the fact that for higher-order methods, the convergence rate of the dispersion error, $2p$, is larger than the convergence rate of the eigenvector, $p+1$.
  \item All methods are significantly affected by a poor mesh quality, although lower-order methods are more sensitive to this than higher-order methods. Flattening the tetrahedra by a factor $10$ reduces the accuracy of the methods by $1$-$2$ orders of magnitude, even though the mesh resolution remains the same and even improves in one direction.
  \item The SIPDG methods are not really sensitive to high P/S-wave velocity ratios, while the accuracy of the higher-order mass-lumped methods reduces slightly when the P/S-wave velocity ratio is increased. The accuracy of the linear mass-lumped method, however, reduces by an order of magnitude when the P/S-wave velocity ratio is raised from $2$ to $10$.
\end{itemize}

\bibliographystyle{abbrv}    
\bibliography{dispersionAnalysis}

\appendix
\section{Stability of the Lax--Wendroff Method}
\begin{thm}
\label{thm:stabLW1}
Consider the following time integration scheme:
\begin{align*}
U(t_{i+1}) = -U(t_{i-1}) + 2\beta U(t_i), &&i=1,2,\dots,
\end{align*}
where $\beta\in\mathbb{R}$ is a constant and $\{U(t_i)\}_{i\geq 0}$ is a sequence of scalars representing a scalar variable $u(t)$ at time slots $t_i=i\Delta t$, with $\Delta t$ the time step size. This scheme is stable, by which we mean that the solution grows at most linearly in time, iff $|\beta|\leq 1$. 
\end{thm}

\begin{proof} If $\beta\in(-1,1)$, then the two independent solutions of the time integration scheme are given by $U(t_n) = e^{\pm\im (t_n \omega)}$, where $\omega$ satisfies $\cos(\omega\Delta t)=\beta$ and $\im:=\sqrt{-1}$ is the imaginary number. Otherwise, if $\beta=1$ (or $\beta=-1$), then the two independent solutions are given by $U(t_n)=1,n$ (or $U(t_n)=(-1)^n, n(-1)^n$). Finally, if $\beta\geq 1$ (or $\beta<-1$), then the two independent solutions are given by $U(t_n)=e^{\pm t_n\omega}$ (or $U(t_n)=-e^{\pm t_n\omega}$), where $\omega$ satisfies $\cosh(\omega\Delta t)=\beta$ (or $-\cosh(\omega\Delta t)=\beta$). Therefore, the scheme grows at most linearly in time iff $\beta\in[-1,1]$.
\end{proof}

\begin{thm}
\label{thm:stabLW2}
Consider the order-$2K$ Lax--Wendroff time integration method given by
\begin{align*}
\vvU(t_{i+1}) = -\vvU(t_{i-1}) + 2\sum_{k=0}^K \frac{1}{(2k)!}\Delta t^{2k}(-M^{-1}A)^k\vvU(t_i), &&i=1,2,\dots,
\end{align*}
where $M$ and $A$ are symmetric positive definite matrices, and $\{\vvU(t_i)\}_{i\geq 0}$ is a sequence of vectors representing a vector variable $\vvu(t)$ at time slots $t_i=i\Delta t$, with $\Delta t$ the time step size. This scheme is stable, by which we mean that the solution grows at most linearly in time, if $\Delta t\leq \sqrt{c_K/\sigma_{max}(M^{-1}A)}$, where $\sigma_{max}(M^{-1}A)$ denotes the spectral radius of $M^{-1}A$ and $c_K$ is defined as
\begin{align*}
c_K &:= \inf\left\{x\geq 0 \;|\; \left|\sum_{k=0}^K \frac{1}{(2k)!}(-x)^k\right| > 1 \right\}.
\end{align*}
\end{thm}

\begin{proof}
We can rewrite the time integration scheme as
\begin{align*}
\vvU(t_{i+1}) = -\vvU(t_{i-1}) + 2B\vvU(t_i),
\end{align*}
where $B:=\sum_{k=0}^K \frac{1}{(2k)!}\Delta t^{2k}(-M^{-1}A)^k$. Since $M$ and $A$ are symmetric positive definite, we can diagonalise $M^{-1}A$ as $VDV^{-1}$, with $D$ a diagonal matrix with only positive real values on the diagonal. We can then diagonalise $B$ as $B=V\left(\sum_{k=0}^K \frac{1}{(2k)!}\Delta t^{2k}(-D)^k\right)V^{-1}$. Using this diagonalisation we can decouple the matrix-vector equations into scalar equations of the form
\begin{align*}
U(t_{i+1}) = -U(_{i-1}) + 2\beta U(t_i), &&i=1,2,\dots,
\end{align*} 
with 
\begin{align*}
\beta&=\sum_{k=0}^K\frac{1}{(2k)!}(-s\Delta t^2)^k &&\text{ for some eigenvalue }s\text{ of }M^{-1}A.
\end{align*}
From the definition of $c_K$, it follows that $|\beta|\leq 1$ for all possible $\beta$, if $\Delta t^2\sigma_{max}(M^{-1}A)\leq c_K$, so if $\Delta t\leq \sqrt{c_K/\sigma_{max}(M^{-1}A)}$. From Theorem \ref{thm:stabLW1} it then follows that this scheme is stable.
\end{proof}

\begin{rem}
\label{rem:stabLW}
The values of $c_K$ can be computed numerically. For example, $c_K=4,12,7.57$ for $K=1,2,3$, respectively.
\end{rem}

\end{document}